\documentclass[12pt]{article}
\usepackage[utf8]{inputenc}
\usepackage{amsmath,amssymb,amsthm, todonotes,url,mathtools}
\usepackage[a4paper, margin=2.5cm]{geometry}
\usepackage{needspace}
\usepackage{libertine} 
\usepackage{inconsolata} 
\usepackage{hyperref}
\hypersetup{hidelinks,
 pdfstartview=FitH,
 colorlinks=true,
 linkcolor=red,
 citecolor=red,
 urlcolor=red}
\usepackage{todonotes}
\usepackage{authblk}
\usepackage{dsfont}
\usepackage{cleveref}
\usepackage[inline]{enumitem}
\setlist[enumerate]{topsep=1ex,itemsep=0pt,partopsep=1ex,parsep=1ex}
\setlist[itemize]{topsep=1ex, itemsep=0pt,partopsep=1ex,parsep=1ex}
\usepackage{tikz}
\usetikzlibrary{decorations}
\usetikzlibrary{decorations.pathmorphing}
\newtheorem{theorem}{Theorem}

\newtheorem{lemma}[theorem]{Lemma}
\newtheorem{corollary}[theorem]{Corollary}

\newtheorem{observation}[theorem]{Observation}

\newtheorem*{claim*}{Claim}
\newtheorem{problem}[theorem]{Problem}
\usepackage{subcaption}

\def\cqedsymbol{\ifmmode$\lrcorner$\else{\unskip\nobreak\hfil
\penalty50\hskip1em\null\nobreak\hfil$\lrcorner$
\parfillskip=0pt\finalhyphendemerits=0\endgraf}\fi}



\makeatletter
\newtheorem*{rep@theorem}{\rep@title}
\newcommand{\newreptheorem}[2]{%
\newenvironment{rep#1}[1]{%
 \def\rep@title{#2 \ref{##1}}%
 \begin{rep@theorem}}%
 {\end{rep@theorem}}}
\makeatother

\newtheorem{conjecture}[theorem]{Conjecture}
\newtheorem*{conjecture*}{Conjecture}
\newreptheorem{conjecture}{Conjecture}

\theoremstyle{definition}                    

\theoremstyle{remark}   

\newtheorem*{remark*}{Remark}

\numberwithin{equation}{section}
\usepackage{soul}

\newcommand{\mytilde}{\raise.17ex\hbox{$\scriptstyle\mathtt{\sim}$}}
\tikzstyle{vertex}=[circle, draw, fill=black, inner sep=0pt, minimum width=4pt]

\title{Infinite induced-saturated graphs}
\author[1]{Marthe Bonamy}
\author[2]{Carla Groenland\footnote{Research supported by the Dutch Research Council (NWO, VI.Veni.232.073)}}
\author[3]{Tom Johnston}
\author[4]{Natasha Morrison\footnote{Research supported by NSERC Discovery Grant RGPIN-2021-02511}}
\author[5]{Alex Scott\footnote{Research supported by EPSRC grant EP/X013642/1}}
\affil[1]{\small CNRS, LaBRI, Universit\'e de Bordeaux, France.}
\affil[2]{\small Delft Institute of Applied Mathematics, TU Delft, the Netherlands.}
\affil[3]{\small University of Bristol, United Kingdom and Heilbronn Institute of Mathematical Research, United Kingdom.}
\affil[4]{\small Department of Mathematics and Statistics, University of Victoria, Canada.}
\affil[5]{\small Mathematical Institute, University of Oxford, United Kingdom.}
\date{\today}
\begin{document}

\maketitle

\begin{abstract}
A graph $G$ is $H$-induced-saturated if $G$ is $H$-free but deleting any edge or adding any edge creates an induced copy of $H$.
There are non-trivial graphs $H$, such as $P_4$, for which no finite $H$-induced-saturated graph $G$ exists. We show that for every finite graph $H$ that is not a clique or an independent set, there always exists a countable $H$-induced-saturated graph. 
In fact, we show that a far stronger property can be achieved: there is a countably infinite $H$-free graph $G$ such that any graph $G'\ne G$ obtained by making a locally finite set of changes to $G$ contains a copy of $H$.
\end{abstract}

\section{Introduction}\label{sec:intro}
For a finite graph $H$, a (finite or countably infinite) graph $G$ is {\em $H$-free} if it does not contain an induced copy of $H$. 
In this paper, we study how ``unstable'' this property can be in the following sense: 
\begin{center}
    If $G$ is $H$-free, must there be other $H$-free graphs that are ``close to'' $G$? 
\end{center}
One natural definition of ``close'' is to say that two graphs $G$ and $G'$ are close if $G'$ can be obtained from $G$ by \textit{perturbing} a single pair of distinct vertices $u, v\in V(G)$. That is, if $uv$ is an edge in $G$, we delete it, and if it is not an edge, we add it.

A graph $G$ on at least two vertices is called \emph{$H$-induced-saturated} if $G$ is $H$-free but perturbing any pair (i.e.~deleting or adding an edge) creates an induced copy of $H$.
We will show that, provided $H$ is not a clique or an independent set, countably infinite $H$-induced-saturated graphs always exist.  In fact, we will show there are graphs that satisfy a vastly stronger property.
A \emph{locally finite perturbation} of $G$ is a graph obtained from $G$ by perturbing a nonempty set of pairs such that every vertex is in a finite number of perturbed pairs; equivalently, we can consider graphs $G'$ such that $E(G)\triangle E(G')$ is the edge set of a locally finite graph.

Our main result is the following.
\begin{theorem}\label{thm:main}
Let $H$ be a finite graph. If $H$ is not a clique or an independent set, then there exists a countably infinite $H$-free graph $G_H$ such that every locally finite perturbation of $G_H$ has an induced copy of $H$.
\end{theorem}
As a (very) special case of this, we obtain the following statement about induced-saturated graphs.
\begin{corollary}
\label{cor:main_basic}
For any finite graph $H$ which is not a clique or independent set, there is a countably infinite $H$-induced-saturated graph $G_H$.
\end{corollary}

Our work sits within the broader theme of research into the geometry of classes of countably infinite graphs.
Theorem \ref{thm:main} tells us that, for $H$ not a clique or independent set, the space of $H$-free graphs has points that are isolated in a very strong sense: there is no locally finite perturbation that gives another point in the space.  This tells us that the space is very poorly connected.

Two countably infinite graphs can also be compared by looking for embeddings between them.
A graph  $U$ is called \textit{strongly universal} (or a faithful universal graph) for a graph class $\mathcal{G}$ if $U\in \mathcal{G}$ and $U$ contains all graphs from $\mathcal{G}$ as an induced subgraph.  For the class of all countable graphs, such a strongly universal graph exists (often named after Rado~\cite{Rado1964}).
A question of particular interest here is for which $H$ there exists a strongly universal graph for the class of countably infinite graphs that do not contain $H$ as a subgraph
(see e.g. \cite{CherlineShelah2007,CherlinShelah2016,CherlinTallgren2007JGT,FurediKomjath97Combinatorica,FurediKomjath1997JGT,Henson1971,komjathMeklerPach1988IJM,KomjathPach84Mathematika}). The existence of universal or strongly universal graphs has also been studied for other graph classes, e.g. planar graphs~\cite{diestelKuhn1999JGT,Pach1981_Ulam} or graphs avoiding $K_n$-subdivisions~\cite{Diestel1985}.

Another important direction within infinite structural graph theory is the classification of countable ultrahomogeneous graphs by Lachlan and Woodrow~\cite{Lachlan1980}, which was followed by classifications of homogeneity in other countable discrete structures (see e.g.~\cite{cameron2006homomorphism,cherlin1998classification,habilthesis,classificationconnectedhomogeneousDigraphs,schmerl1979countable}). Similarly, Diestel~\cite{diestel1990survey} classified which countable graphs have a prime decomposition and Thomassen~\cite{Thomassen1982duality} proved various classifications for infinite graphs relating to the existence of dual graphs. For surveys on infinite graphs, see e.g.~\cite{Diestel2017survey,Komjath2011Survey,Stein2011survey}.

\paragraph{Finite induced-saturated graphs}
Finite induced-saturated graphs and variations on this have been studied extensively in recent years (see e.g. \cite{AxenovichCsikos19,BehrensErbesSantanaYagerYeager16, ChoChoiParkPaths,DvorakPaths,fan2025halfway,MartinSmith12,RatyP6,Tennenhouse16}). It is immediately clear that there is no (finite or infinite) $H$-induced-saturated graph when $H$ is a clique or independent set with more than two vertices.\footnote{If $H=K_t$ where $t\ge3$, then no $H$-induced-saturated graph can contain an edge (as $H$-free graphs remain $H$-free under edge deletion) while graphs with no edges also fail to be $H$-induced-saturated (as adding a single edge does not create a copy of $H$). The argument for independent sets is analogous.}
But the picture is more complicated for more general graphs.  For example,
one of the first results on induced-saturated graphs showed that there is no finite $P_4$-induced-saturated graph \cite{MartinSmith12}. 
This result can also be shown by noting that every finite $P_4$-free graph on at least two vertices contains vertices $u, v$ that are twins (that is, they have the same neighbours, except possibly each other). Perturbing the pair $uv$ does not create a copy of $P_4$ in this case, and so the resulting graph is still $P_4$-free.
However, it is straightforward to construct examples of $P_t$-induced-saturated graphs for $t=2,3$ and, after some partial progress by various authors \cite{ChoChoiParkPaths,RatyP6,Tennenhouse16}, Dvořák~\cite{DvorakPaths} showed that $P_t$-induced-saturated graphs exist for all $t\geq 6$.  Shortly afterwards, the authors found examples of $P_5$-induced-saturated graphs using a computer search \cite{bonamy2020induced} (one of which is shown in Figure~\ref{fig:finitesat}).
Thus, a finite $P_t$-induced-saturated graph exists if and only if $t \neq 1,4$.
\begin{figure}
    \centering
\begin{tikzpicture}

\node[vertex] (0) at (0:2) {};
\node[vertex] (1) at (-36:2) {};
\node[vertex] (2) at (36:2) {};
\node[vertex] (3) at (-72:2) {};
\node[vertex] (4) at (72:2) {};
\node[vertex] (5) at (48: 1) {};
\node[vertex] (6) at (180:2) {};
\node[vertex] (7) at (228:1) {};
\node[vertex] (8) at (-108:2) {};
\node[vertex] (9) at (144:2) {};
\node[vertex] (10) at (216:2) {};
\node[vertex] (11) at (108:2) {};

 \path[draw, thick]
(0) edge node {} (1) 
(0) edge node {} (2) 
(0) edge node {} (3) 
(0) edge node {} (4) 
(0) edge node {} (5) 
(0) edge node {} (6) 
(1) edge node {} (2) 
(1) edge node {} (3) 
(1) edge node {} (7) 
(1) edge node {} (8) 
(1) edge node {} (9) 
(2) edge node {} (4) 
(2) edge node {} (7) 
(2) edge node {} (10) 
(2) edge node {} (11) 
(3) edge node {} (5) 
(3) edge node {} (8) 
(3) edge node {} (10) 
(3) edge node {} (11) 
(4) edge node {} (5) 
(4) edge node {} (8) 
(4) edge node {} (9) 
(4) edge node {} (11) 
(5) edge node {} (7) 
(5) edge node {} (9) 
(5) edge node {} (10) 
(6) edge node {} (7) 
(6) edge node {} (8) 
(6) edge node {} (9) 
(6) edge node {} (10) 
(6) edge node {} (11) 
(7) edge node {} (8) 
(7) edge node {} (11) 
(8) edge node {} (10) 
(9) edge node {} (10) 
(9) edge node {} (11);

\end{tikzpicture}
\caption{The complement of the icosahedral graph is induced-saturated for $P_5$.}
    \label{fig:finitesat}
\end{figure}
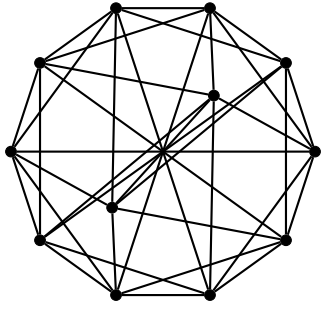

The motivation behind the study of induced-saturated graphs stems from a long history of research concerning ``saturated'' graphs, initiated by Erd\H{o}s, Hajnal and Moon~\cite{EHM}. A graph is said to be $H$-\emph{saturated} if it contains no subgraph isomorphic to $H$ and the addition of any edge creates such a subgraph. Note that, in contrast to the induced setting, removing an edge can never create a subgraph isomorphic to $H$ where one did not exist before. 
It is easy to construct an $H$-saturated graph: start from an empty graph and keep adding edges as long as the chosen edge does not create a subgraph isomorphic to $H$. If there are no more edges that can be added, the graph must be $H$-saturated.
While the question of existence is straightforward, determining the minimum and maximum number of edges in an $H$-saturated graph of given size is much more interesting, and it has been extensively studied (see the comprehensive survey of Currie, Faudree, Faudree and Schmitt~\cite{faudree2011survey} for a review of literature on saturation in graphs and related variants). 

\medskip

\paragraph{Paper overview}
The rest of the paper is organized as follows. In the next section, we discuss some of the key ingredients in the proof of Theorem~\ref{thm:main}.
We discuss notation and conventions in Section~\ref{notation}, and give the proof of Theorem~\ref{thm:main} in the following sections. We conclude with a discussion of directions for further work.

\section{Some proof ideas}\label{informal}

In this section we discuss some of the ideas from the proof of Theorem~\ref{thm:main}.
We first describe how one might hope to deal with individual perturbations in the context of Corollary~\ref{cor:main_basic} before outlining how this can be adapted to handle locally finite perturbations.  
The proof of Theorem~\ref{thm:main} is given in later sections, which do not rely on the informal discussion in this section.

\subsection{\texorpdfstring{$P_4$}{P\textunderscore4} and fixing operations}\label{subsec:P4}
The case of $P_4$-free graphs is particularly interesting as there is no finite $P_4$-induced-saturated graph. 

We find an infinite $P_4$-induced-saturated graph $G$ by considering a sequence $(G_i)_{i\in \mathbb{N}}$ of graphs and taking its limit. 
We define the sequence $(G_i)$ inductively.  Throughout the construction, we
keep a list of pairs which are ``bad'', in the sense that perturbing them does not create a copy of $P_4$. We obtain $G_{i+1}$ from $G_i$ by attaching a gadget that ``fixes'' the next bad pair in the list without creating a copy of $P_4$. We then update the list by adding the bad pairs involving a new vertex to the end of the list. 

When there is a simple fixing operation for $H$, this gives us a straightforward recipe for proving Corollary~\ref{cor:main_basic}. 
In the case of $P_4$, the construction is as follows.  Given a bad pair $\{x,y\}$, we create duplicates (or ``twins'') $x'$ and $y'$ of $x$ and $y$.   
Since $P_4$ has no twins, adding a twin can never create an induced copy of $P_4$. Moreover, it is not hard to see that in the resulting graph perturbing the pair $\{x,y\}$ creates a  copy of $P_4$. The sequence $(G_i)_{i\in \mathbb{N}}$ of graphs is obtained by considering an arbitrary $P_4$-free graph, computing a list of its bad pairs, then obtaining $G_{i+1}$ from $G_i$ by fixing its first bad pair and updating the list of bad pairs.

We give an example in Figure \ref{fig:p4}: we begin with a graph $G_1$ with a single bad pair $\{u_1,u_2\}$, and fix it to obtain $G_2$ with two bad pairs $\{\{u_1,v_1\},\{u_2,v_2\}\}$.  Fixing $\{u_1, v_1\}$ creates a graph $G_3$ with three bad pairs, with $\{u_2,v_2\}$ at the front of the list. 

\begin{figure}[ht!]
    \centering
 \begin{tikzpicture}[scale=1]
        \draw (0,0) node(a)[style=vertex]{};
        \draw (0,0) node[left]{$u_2$};
        \draw (0,1.5) node(b)[style=vertex]{};
        \draw (0,1.5) node[left]{$u_1$};

        \draw[red] (a) -- (b);

        \draw (3,0) node(c)[style=vertex]{};
        \draw (3,0) node[left]{$u_2$};
        \draw (3,1.5) node(d)[style=vertex]{};
        \draw (3,1.5) node[left]{$u_1$};
        \draw (4.5,0) node(e)[style=vertex]{};
        \draw (4.5,0) node[right]{$v_2$};
        \draw (4.5,1.5) node(f)[style=vertex]{};
        \draw (4.5,1.5) node[right]{$v_1$}; 

        \draw[dashed, red] (d) -- (f);
        \draw[dashed, red] (c) -- (e);
        \draw (d) -- (e);
        \draw (d) -- (c);
        \draw (c) -- (f);
        \draw (f) -- (e);

        \draw (7.5,0) node(u2)[style=vertex]{};
        \draw (7.5,0) node[left]{$u_2$};
        \draw (7.5,1.5) node(u1)[style=vertex]{};
        \draw (7.5,1.5) node[left]{$u_1$};
        \draw (9,0) node(v2)[style=vertex]{};
        \draw (9,0) node[right]{$v_2$};
        \draw (9,1.5) node(v1)[style=vertex]{};
        \draw (9,1.5) node[right]{$v_1$}; 
        \draw (9.5,0.75) node(x1)[style=vertex]{};
        \draw (9.5,0.75) node[right]{$x_1$}; 
        \draw (8.25,2) node(w1)[style=vertex]{};
        \draw (8.25,2) node[right]{$w_1$}; 

        \draw[dashed, red] (u2) -- (v2);
        \draw[red] (u1) -- (w1);
        \draw[red] (v1) -- (x1);
        \draw (u1) -- (u2);
        \draw (u1) -- (v2);
        \draw (w1) -- (v2);
        \draw (x1) -- (v2);
        \draw (x1) -- (u2);
        \draw (u2) -- (v1);
        \draw (w1) -- (u2);
        \draw (v1) -- (v2);

\end{tikzpicture}
        \caption{A sequence of graphs $G_1,G_2,G_3,\dots$ is obtained by repeatedly applying a fixing operation. The red edges and red non-edges are not fixed.}
    \label{fig:p4}
\end{figure}
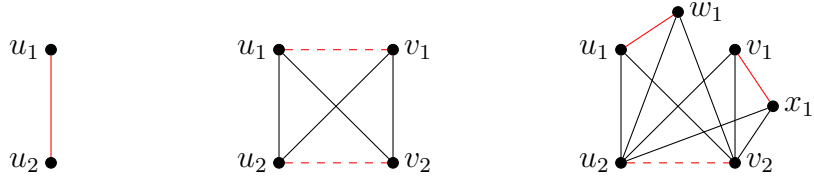

 We note that, for any $i$, and for any bad pair $\{x,y\}$ of a graph $G_i$, there is a (later) index $j$ such that perturbing $\{x,y\}$ in $G_j$ or any subsequent graph yields an induced $P_4$. In other words, any bad pair is ultimately fixed. Therefore, there is a countably infinite $P_4$-free graph $G_{P_4}$ such that perturbing any pair in $G_{P_4}$ results in an induced copy of $P_4$. That is, Corollary~\ref{cor:main_basic} is true for $H=P_4$ and, more generally, for any graph $H$ that admits fixing operations for bad pairs.

\subsection{\texorpdfstring{$C_5$}{C\textunderscore5} and gatekeepers}\label{subsec:C5}
The fixing operations for $P_4$ defined above rely heavily on the strong structural properties of $P_4$-free graphs. Since we need to prove Corollary~\ref{cor:main_basic} for every graph $H$ that is not a clique or independent set, we will need something much more general. The naïve way to define a fixing operation for $H$ is by gluing a copy of $H-e$ with the missing edge aligned with a bad non-edge (or a copy of $H+e$ with the extra edge aligned with a bad edge). While this guaranteea that perturbing the bad pair will result in an induced copy of $H$, the pitfall is that the gluing might create a copy of $H$. One may hope to avoid this by choosing carefully which edge to add or delete in $H$, but this may not always be possible (consider $H = P_4$). 

However, for $H=C_5$, this \emph{is} a good approach (see Figure \ref{fig:c5}). To fix a bad non-edge $xy$, we may glue $C_5-e$ by adding three vertices $u_1,u_2,u_3$ and four edges $xu_1,u_1u_2,u_2u_3,u_3y$. Note that the resulting graph is $C_5$-free, as every induced copy of $C_5$ involving some new edge would involve all of them. Similarly, to fix a bad edge $xy$, we may glue $C_5+e$ by adding three vertices $u_1,u_2,u_3$ and five edges $xu_1,u_1y,xu_2,u_2u_3,u_3y$. It is not hard to see that the resulting graph is also $C_5$-free. 
\begin{figure}[ht!]
    \centering
    \begin{tikzpicture}
        \draw [draw=black] (0,0) rectangle (2,2);
        \draw (1.5,0.5) node(a)[style=vertex]{};
        \draw (1.5, 0.5) node[left]{$y$};
        \draw (1.5,1.5) node(b)[style=vertex]{};
        \draw (1.5, 1.5) node[left]{$x$};
        \draw[blue] (2.5,0.5) node(c)[style=vertex, fill=blue]{};
        \draw (2.5, 0.5) node[below]{$u_3$};
        \draw[blue] (2.5,1.5) node(d)[style=vertex, fill=blue]{};
        \draw (2.5, 1.5) node[above]{$u_1$};
        \draw[blue] (3, 1) node(e)[style=vertex, fill=blue]{};
        \draw (3, 1) node[right]{$u_2$};
        
        \draw[dashed] (a)  -- (b);
        \draw[blue] (a)  -- (c);
        \draw[blue] (b)  -- (d);
        \draw[blue] (c)  -- (e);
        \draw[blue] (d)  -- (e);

        \draw [draw=black] (6,0) rectangle (8,2);
        \draw (7.5,0.5) node(a)[style=vertex]{};
        \draw (7.5, 0.5) node[left]{$y$};
        \draw (7.5,1.5) node(b)[style=vertex]{};
        \draw (7.5, 1.5) node[left]{$x$};
        \draw[blue] (8.5,0.5) node(c)[style=vertex, fill=blue]{};
        \draw (8.5, 0.5) node[below]{$u_3$};
        \draw[blue] (8.5,1.5) node(d)[style=vertex, fill=blue]{};
        \draw (8.5, 1.5) node[above]{$u_1$};
        \draw[blue] (9, 1) node(e)[style=vertex, fill=blue]{};
        \draw (9, 1) node[right]{$u_2$};

        \draw (a)  -- (b);
        \draw[blue] (a)  -- (c);
        \draw[blue] (a)  -- (d);
        \draw[blue] (b)  -- (e);
        \draw[blue] (b)  -- (d);
        \draw[blue] (c)  -- (e);

    \end{tikzpicture}
     \caption{The fixing operation for non-edges (left) and edges (right) is depicted. The blue vertices represent newly created vertices.}
     \label{fig:c5}
 \end{figure}
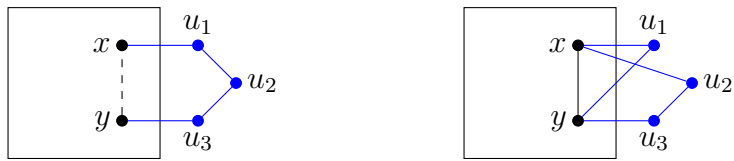
In fact, this idea works for many graphs, including every $3$-connected graph $H$ that is not a clique. Indeed, if we ``fix'' the pair $\{x,y\}$ in a similar fashion to Figure~\ref{fig:c5}, then $\{x,y\}$ forms a 2-cut in the resulting graph between the ``old'' and ``new'' vertices. If $H$ is 3-connected, then it cannot contain vertices of both types so no new copies of $H$ can be created.
This proves Corollary~\ref{cor:main_basic} for any $3$-connected graph $H$.

We now formalise what makes $C_5$ and 3-connected graphs (but not $P_4$) behave well.
An edge $uv$ in a graph $H$ is a \emph{gatekeeper} if for any $H$-free graph $G$, gluing a copy of $H - uv$ on a non-edge of $G$ results in an $H$-free graph. A similar definition holds for non-edges.
We say $H$ admits a gatekeeper \emph{of each type} if it contains both an edge and a non-edge that are gatekeepers. 

A graph $H$ that admits a gatekeeper of each type admits a fixing operating for bad edges as well as for bad non-edges, using the naïve approach described above.
Note that a path contains no edge that is a gatekeeper. 

\subsection{\texorpdfstring{$C_5$}{C\textunderscore5} with a leaf and cores}\label{subsec:cores}
Unfortunately, even small modifications to a graph can make a graph which is much harder to handle. For example, consider the graph obtained from $C_5$ by adding a pendant edge to one of the vertices. 
While we easily found fixing operations for $C_5$ using gatekeepers, this new graph does not admit gatekeepers of either type. 
Instead, we focus on the $C_5$ and work within the class of $C_5$-free graphs.
If we work within this class, we can use the gatekeepers for $C_5$ and not create a copy of $C_5$ or of our target graph. 
These gatekeepers are not quite enough to guarantee a perturbation creates a copy of $C_5$ plus a leaf, but this is easily solved by adding a leaf to every vertex we add.
Note that the leaves cannot be in a $C_5$, so this is a valid modification.
We now formalise and extend this idea using the notion of cores.

The $2$-core $H^*$ of a finite graph $H$ is obtained by iteratively removing vertices of degree at most 1. Crucially, if we take a graph $G$ which is $H^*$-free, then adding a leaf to each vertex of $G$ does not create a copy of $H^*$ let alone $H$, nor does adding an isolated vertex. Given a fixing operation for $H^*$, we can create a fixing operation for $H$ as follows. Throughout the process, we maintain that the graph created so far is $H^*$-free. We interchange fixing operations for $H^*$ with steps that either add a leaf to each vertex or add an isolated vertex. Since every bad pair is eventually fixed by the fixing operation for $H^*$, if we perturb a pair then we can embed a copy of $H^*$. We can extend our copy of $H^*$ into a copy of $H$ as we have added infinitely many leaves and isolated vertices. 

Therefore, if a graph $H^*$ admits fixing operations, any graph $H$ whose $2$-core is $H^*$ admits fixing operations. We can generalise this notion by iteratively deleting any vertex with at most $k$ neighbours or at most $\ell$ non-neighbours. 

In particular, the 3-core $H^*$ of a graph $H$ is obtained by iteratively removing vertices of degree at most 2. By the discussion above, Corollary~\ref{cor:main_basic} holds for a graph $H$ if its 3-core is 3-connected.  
We are also done when the complement $\overline{H}$ of $H$ is 3-connected: the complement of an $\overline{H}$-induced-saturated graph is $H$-induced-saturated. When $H$ has a 2-cut and both ``sides'' contain at least three vertices, then $\overline{H}$ is close to having a 3-connected 3-core. However, there is one ``bad'' 2-cut to take into account, depicted in Figure~\ref{fig:badcut}. 
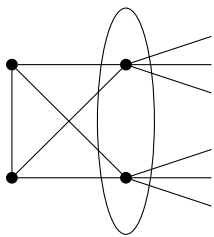
\begin{figure}[ht!]
    \centering
    \begin{tikzpicture}[scale=1.5]
        \draw (1.5,0.5) node(a)[style=vertex]{};
        \draw (1.5,1.5) node(b)[style=vertex]{};
        \draw (0.5,0.5) node(c)[style=vertex]{};
        \draw (0.5,1.5) node(d)[style=vertex]{};
        
        \draw (c)  -- (d);
        \draw (a)  -- (c);
        \draw (b)  -- (d);
        \draw (a)  -- (d);
        \draw (b)  -- (c);

        \draw (1.5,1) ellipse (0.25cm and 1cm);
        \draw (1.5,1.5)   -- (2.25,1.25);
        \draw (1.5,1.5)   -- (2.25,1.5);
        \draw (1.5,1.5)   -- (2.25,1.75);
        \draw (1.5,0.5)   -- (2.25,0.25);
        \draw (1.5,0.5)   -- (2.25,0.5);
        \draw (1.5,0.5)   -- (2.25,0.75);

    \end{tikzpicture}
    \caption{A ``bad'' 2-cut is shown. No such cuts are present in the $3^*$-core.}
    \label{fig:badcut}
\end{figure}

We define the \textit{$3^*$-core} of a graph $H$ as the graph obtained from $H$ by iteratively removing vertices of degree at most $2$, as well as true twins of degree $3$ (``bad'' 2-cuts).
As before, fixing operations for a graph $H^*$ can be extended to any graph $H$ whose $3^*$-core is $H^*$. 
If a graph $H$ has a non-empty $3^*$-core, then either the $3^*$-core of $H$ or the $3^*$-core of $\overline{H}$ is 3-connected. There is a small catch: the $3^*$-core could be a clique and this will give a few more cases to handle.

\subsection{Outline of the proof}\label{sec:combining}
\noindent We can now give a broad overview of the steps in our proof. The proof will be given formally in Section~\ref{subsec:proof}.
\begin{itemize}
    \item We give a fixing operation for all $H$ whose $3^*$-core is 3-connected and not a clique. The reader should already be able to verify this for the weaker Corollary~\ref{cor:main_basic} from the discussion above. We provide the stronger statements needed for Theorem~\ref{thm:main} in Section~\ref{sec:3connected}.
    \item We give a fixing operation for all $H$ whose 2-core is a $K_{1,1,p}$ or $K_{2,p}$ and a direct construction for all forests with a unique vertex of maximum degree in Section \ref{sec:large}.
    \item We prove a structure theorem (Theorem \ref{thm:cases}) that shows that for all finite graphs $H$ on at least 12 vertices, either $H$ or its complement falls into one of the above categories. This is done in Section \ref{sec:charact}.
    \item We provide direct constructions for specific graphs $H$ on at most 7 vertices in Section \ref{sec:small}. 
    \item We use a computer to check that for all graphs $H$ on at most 11 vertices, either $H$ or its complement falls into one of the above categories. This is detailed in Section \ref{sec:computermagic}.
\end{itemize}

\subsection{From individual perturbations to locally finite perturbations}\label{subsec:strongly}
For Corollary~\ref{cor:main_basic}, we only need to generate a copy of $H$ when we make a single perturbation: we either turn an edge into a non-edge or vice versa. For Theorem \ref{thm:main}, we allow many pairs to be perturbed at once, and the arguments above no longer directly apply. 

We describe one way in which we show that a pair has been fixed. After a locally finite perturbation, we want to embed $H$ vertex-by-vertex. Suppose that we already embedded $h_1,\dots,h_s$. If there are infinitely many options from which to choose the next vertex $h_{s+1}$, then there is always an option for which the adjacencies to $h_1,\dots,h_s$ have not been altered by the locally finite perturbation. Indeed, for each $h_i$, there are only finitely many $v$ so that $\{h_i,v\}$ has been altered. One way to generate infinitely many options is to blow up vertices into infinite cliques. This does not always work, since it may create a copy of $H$. Moreover, it also creates new edges to be ``fixed''. Nevertheless, the difference between allowing a single perturbation or allowing any locally finite perturbation is reasonably small in most of our proofs. In some places, it does require additional insights, such as for 3-connected graphs (handled by Lemma~\ref{lem:blowup}). 

There is one small additional point that comes up when we prove our stronger variant. Assume that a graph $H$ admits fixing operations as in Section~\ref{subsec:P4}. For Corollary~\ref{cor:main_basic}, we start from a finite graph, and fixing a pair adds finitely many vertices, hence finitely many new bad pairs. Each graph of the sequence is finite, and new bad pairs can simply be added to the end of the list of existing bad pairs. 
However, in the proof of  Theorem~\ref{thm:main}, we add countably infinitely many vertices to fix a single bad pair, thus possibly infinitely many new bad pairs. Since we add countably many bad pairs for a countable number of times, it is possible to schedule the fixes in a careful order by a standard enumeration argument; this is described in the proof of Lemma~\ref{lem:fix_to_strong}.

\section{Notation and definitions}\label{notation}
The word ``graph'' in this paper allows for infinite graphs, but the vertex set is always required to be countable. When we search for a  (strongly) $H$-induced-saturated graph, $H$ will always refer to a finite graph. We only consider simple graphs (without self-loops, directed edges or parallel edges).

A \textit{clique} is a set of vertices that are pairwise adjacent and a \textit{independent set} is a set of vertices that are pairwise non-adjacent. A connected graph $G=(V,E)$ is \emph{$k$-connected} if the graph has at least $k+1$ vertices
and remains connected after removing any $k-1$ vertices. For graphs $G$ and $H$, we say that $G$ is \textit{$H$-free} if $G$ does not contain $H$ as induced subgraph. 

Write $N(v)$ for the neighbourhood of $v$, that is, the set of vertices adjacent to $v$, and write $N[v]=N(v)\cup \{v\}$ for the closed neighbourhood. 
We say two vertices $u,v$ are \emph{true twins} if $N[u]=N[v]$ and \emph{false twins} if $N(u)=N(v)$ (and so $u$ and $v$ are not adjacent). Write $uv$ to denote a pair of vertices $\{u,v\}$ with $u\neq v$. Write $P_t$ for the path on $t$ vertices and $K_t$ for the complete graph on $t$ vertices.

The definitions below are similar to those in the introduction, but now formally stated in full generality.

\paragraph{Locally finite perturbation}
For a pair of distinct vertices $u,v\in V(G)$, \emph{perturbing} the pair $uv$ is the operation which removes $uv$ from $E(G)$ if it is present and adds it if it is not present. A \emph{locally finite perturbation} of $G$ is a graph obtained from $G$ by perturbing arbitrarily many pairs (at least one, possibly infinitely many) under the constraint that for any vertex $v$, the number of perturbed pairs involving $v$ is finite. 

\paragraph{Strongly saturating}
For a finite graph $H$, say a graph $G$ is \textit{strongly $H$-induced-saturated} if $G$ does not contain $H$ as an induced subgraph, but any locally finite perturbation of $G$ does. Note that the very notion of strongly $H$-induced-saturated requires $G$ to be infinite. We say that $G$ is \textit{$H$-induced-saturated} if $G$ does not contain $H$ as an induced subgraph, but a copy of $H$ is created when an edge is added to $G$ or removed from $G$.

\paragraph{Fixed and unfixed edges}
Say a pair $xy$ in a graph $G$ is \emph{unfixed} (for $H$) if there is a locally finite perturbation which perturbs the pair $xy$ and does not result in an induced copy of $H$. Otherwise we call the pair \textit{fixed} (for $H$). Note that a graph $G$ is strongly $H$-induced-saturated if and only if all pairs of $G$ are fixed for $H$.

\paragraph{Gluing} 
Given two graphs $G$ and $G'$ on disjoint vertex sets with $A\subseteq V(G)$ and $A' \subseteq V(G')$, and a bijection $f:A'\to A$, the graph $G''$ obtained from \textit{gluing} $G'$ on $G$ along $f$ has $V(G'')=V(G)\cup V(G')\setminus A'$ and $uv\in E(G'')$ for distinct $u,v\in V(G'')$ if and only if 
\begin{enumerate}
    \item $uv\in E(G)$,
    \item $uv\in E(G')$,
    \item $u\in A$ and $f^{-1}(u)v\in E(G')$, or
    \item $u,v\in A$ and $f^{-1}(u)f^{-1}(v)\in E(G')$.
\end{enumerate}
We will most commonly apply this operation for $A=\{u,v\}$, $A'=\{u',v'\}$ and $f(u)=u',f(v)=v'$. An example of this is given in Figure~\ref{fig:glue}.
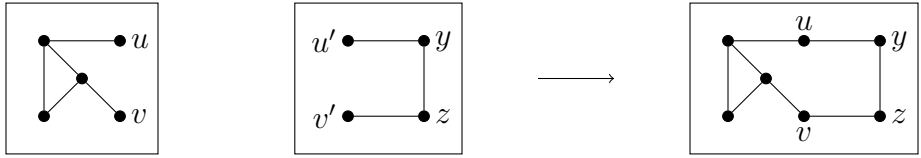
\begin{figure}
   \centering
    \begin{tikzpicture}
        \draw [draw=black] (0,0) rectangle (2,2);
        \draw (1.5,0.5) node(a)[style=vertex]{};
        \draw (1.5, 0.5) node[right]{$v$};
        \draw (1.5,1.5) node(b)[style=vertex]{};
        \draw (1.5, 1.5) node[right]{$u$};
        \draw (0.5,0.5) node(c)[style=vertex]{};
        \draw (0.5,1.5) node(d)[style=vertex]{};
        \draw (1, 1) node(e)[style=vertex]{};
        
        \draw (c)  -- (d);
        \draw (b)  -- (d);
        \draw (c)  -- (e);
        \draw (d)  -- (e);
        \draw (a)  -- (e);

        \draw [draw=black] (3.8,0) rectangle (6,2);
        \draw (5.5,0.5) node(z)[style=vertex]{};
        \draw (5.5, 0.5) node[right]{$z$};
        \draw (5.5,1.5) node(y)[style=vertex]{};
        \draw (5.5, 1.5) node[right]{$y$};
        \draw (4.5,0.5) node(v)[style=vertex]{};
        \draw (4.5, 0.5) node[left]{$v'$};        
        \draw (4.5,1.5) node(u)[style=vertex]{};
        \draw (4.5, 1.5) node[left]{$u'$};        

        \draw (y)  -- (z);
        \draw (v)  -- (z);
        \draw (u)  -- (y);
        \draw[->]        (7,1)   -- (8,1);

        \draw [draw=black] (9,0) rectangle (12,2);
        \draw (9.5,0.5) node(1)[style=vertex]{};
        \draw (9.5,1.5) node(2)[style=vertex]{};
        \draw (10,1) node(3)[style=vertex]{};
        \draw (10.5,1.5) node(4)[style=vertex]{};
        \draw (10.5, 0.5) node[below]{$v$};        
        \draw (10.5, 1.5) node[above]{$u$};        
        \draw (10.5,0.5) node(5)[style=vertex]{};
        \draw (11.5,0.5) node(6)[style=vertex]{};
        \draw (11.5, 0.5) node[right]{$z$};        

        \draw (11.5,1.5) node(7)[style=vertex]{};
        \draw (11.5, 1.5) node[right]{$y$};        

        \draw (1)  -- (2);
        \draw (1)  -- (3);
        \draw (2)  -- (3);
        \draw (3)  -- (5);
        \draw (2)  -- (4);
        \draw (4)  -- (7);
        \draw (5)  -- (6);
        \draw (6)  -- (7);
        
    \end{tikzpicture}
    \caption{An example is given on how two graphs can be glued on non-edges $\{u,v\}$ and $\{u',v'\}$.}
    \label{fig:glue}
\end{figure}

\paragraph{Fixing operations}
We will sometimes need to maintain the stronger property that the graph is not just $H$-free but that it also does not contain a copy of some core of $H$. For this reason, we define a fixing operation relative to a class of graphs, as follows.

Given a class of graphs $\mathcal{F}$, we say the finite graph $H$ admits an \emph{edge fixing operation} (resp. \emph{non-edge fixing operation}) for $\mathcal{F}$ if for every $G\in \mathcal{F}$ and every edge (resp. non-edge) $xy$ of $G$, there exists a graph $G'$ obtained from gluing a graph onto $G$  such that
\begin{itemize}
\item $G'\in \mathcal{F}$, and
\item any locally finite perturbation of $G'$ which perturbs $xy$ contains a copy of $H$.
\end{itemize}
Crucially, $G$ is an induced subgraph of $G'$.
We say $H$ admits a \textit{fixing operation} for $\mathcal{F}$ if it admits both an edge fixing operation and a non-edge fixing operation.
We show in Lemma \ref{lem:fix_to_strong} that if $H$ admits a fixing operation for a non-empty class $\mathcal{F}$ of $H$-free graphs, then there exists a strongly $H$-induced-saturated graph.

\paragraph{Gatekeepers} An edge (resp. non-edge) $uv$ in a graph $H$ is a \emph{gatekeeper} if for any $H$-free graph $G$, gluing a copy of $H - uv$ on a non-edge (resp. edge) of $G$ results in an $H$-free graph.

\paragraph{Cores} Given a graph $H$, the \emph{$(k, \ell)$-core} is the graph obtained by iteratively removing vertices which have fewer than $k$ neighbours or less than $\ell$ non-neighbours. 
We write \textit{$k$-core} to refer to the $(k,0)$-core. 

The \textit{$3^*$-core} of a graph $H$ is obtained by iteratively removing vertices which have at most $2$ neighbours and removing pairs $uv$ of vertices with $|N[u]\cup N[v]|\leq 4$. 

\section{Graphs with a 3-connected core}
\label{sec:3connected}
In this section we prove Theorem~\ref{thm:main} in the case where $H$ has a 3-connected $3^*$-core which is not a clique.

In order to handle locally finite perturbations (rather than single edge perturbations), we will need the following result. 
Given a graph $H$ with a pair of distinct vertices $uv$, we define the infinite graph $H_{uv\to c}$ (resp. $H_{uv\to s}$) as the graph obtained from $H$ by perturbing the pair $uv$ and then blowing up every vertex other than $u$ and $v$ into an infinite clique (resp. independent set). When we blow up a vertex $v$ into an infinite clique (or independent set), we replace the vertex $v$ by new vertices inducing an infinite clique (or independent set) and add edges from all of new vertices to all of the vertices that used to be adjacent to $v$.
\begin{lemma}
\label{lem:blowup}
Let $H$ be a finite graph which is not a clique or independent set. Then $H$ has an edge $uv$ such that one of $H_{uv\to c}$ and $H_{uv\to s}$ contains no induced copy of $H$. 
\end{lemma}
\begin{proof} 
First, suppose that $uv$ is an edge in $H$ such that 
neither $u$ nor $v$ has a true twin. We will show that $H_{uv\to c}$ contains no induced copy of $H$. Note that the property of being true twins is an equivalence relation and contracting the equivalence classes gives a graph $H'$ which has no true twins. This leaves the vertices $u$ and $v$ unchanged as they have no true twins. Note that $H'$ is an induced subgraph of $H$ so if there is an induced copy of $H$ in $H_{uv\to c}$, then there must be a copy of $H'$ as well. Since $H'$ has no true twins, every vertex of $H'$ must be in the blow-up of a different vertex in $H_{uv\to c}$. But $H_{uv\to c}$ is the same as the graph obtained by blowing up each vertex of $H' - uv$ (instead of $H-uv$) except for $u$ and $v$ into an infinite clique. So a copy of $H'$ in $H_{uv\to c}$ gives an induced copy of $H'$ in $H'-uv$, a contradiction. 

A similar argument shows that if neither $u$ and $v$ contain a false twin, the graph $H_{uv\to s}$ contains no induced copy of $H$. 

Next, we argue that a vertex $v$ cannot have both a true twin $x$ and a false twin $y$. Indeed, $x$ is adjacent to $v$, so $y$ should also be adjacent to $x$ (since $v,y$ are false twins). But $y$ is not adjacent to $v$ so $x$ should not be adjacent to $y$. This gives a contradiction. 

So we may and will assume that for each edge $ab$ of $H$, either $a$ has a false twin and $b$ a true twin, or vice versa (as otherwise we are done). This means we can partition the vertices of $H$ into three sets, a set $A$ of vertices with a true twin, a set $B$ of vertices with a false twin $B$ and a set $C$ of vertices with neither a true twin nor a false twin. Since for each edge $ab$ of $H$, either $a$ has a false twin and $b$ a true twin, or vice versa, edges can only go between $A$ and $B$. In particular, there are no edges in $A$. But any vertex in $A$ must have an edge to its true twin in $A$ so $A$ is empty. But then there are no edges at all in our graph, contradicting our assumption that $H$ is not an independent set.
\end{proof}

Since the complement of an independent set or clique is again an independent set or clique, we immediately get the following corollary by taking complements.

\begin{corollary}
\label{cor:blowup}
For every non-trivial finite graph $H$, there is a non-edge $uv$ such that one of $H_{uv\to c}$ and $H_{uv\to s}$ contains no induced copy of $H$. 
\end{corollary}

We next repeat two simple, but important, observations.
\begin{observation}
\label{obs:3connglue}
    Let $H$ be a finite 3-connected  graph and let $G$ and $G'$ be two graphs which are $H$-free. Let $G''$ be obtained by gluing the graphs $G$ and $G'$ together along either an edge or a non-edge $xy$. Then $G''$ is $H$-free.
\end{observation}
Indeed, any copy of $H$ would need to contain both a vertex $u \in V(G) \setminus \{x,y\}$ and a vertex $v \in V(G') \setminus \{x,y\}$, but we can disconnect $u$ and $v$ in $G''$ by removing the vertices $x$ and $y$. This gives a vertex cut of $H$ of size 2, contradicting the assumption that $H$ is 3-connected. In the terminology introduced above, the observation says that any two vertices in $H$ are gatekeepers.
\begin{observation}
\label{lem:blow_up_property}
Let $H$ be a finite graph with a pair of distinct vertices $uv$. Then, for $G'\in \{H_{uv\to c},H_{uv\to s}\}$, any locally finite perturbation of $G'$ which perturbs $uv$ contains a copy of $H$.
\end{observation}
We repeat the argument here for completeness. The perturbation brings the pair $uv$ to the same status it had in $H$; since the perturbation is locally finite, we can embed a copy of $H$ vertex-by-vertex, as follows. We send $u,v$ in $H$ to $u,v$ in $G'$ respectively. We then iterate through the remaining vertices of $H$ and map each vertex to one of the vertices in the blow-up of this vertex in $G'$. There are infinitely many options in each blow-up, so at least one of these vertices is not incident with the finite number of perturbations incident to the finite number of vertices we have embedded so far, and we map $h$ to any such vertex.

We can now define our fixing operation and will then extend this to handle cores.
\begin{lemma}
\label{lem:3-conn}
If $H$ is a finite graph which is 3-connected and not a clique, then $H$ admits a fixing operation for the class of $H$-free graphs.
\end{lemma}
\begin{proof}
Suppose that $H$ is a finite 3-connected graph and we are given an $H$-free graph $G$ and an unfixed edge $xy$ of $G$.
By Lemma \ref{lem:blowup}, we may assume that $H$ has an edge $uv$ such that there is a graph $G'\in \{H_{uv\to c},H_{uv\to s}\}$ that is $H$-free. 
Let $G''$ be the graph formed by gluing together $G$ and $G'$ along the edge $xy$ of $G$ and the edge $uv$ of $G'$.  By Observation \ref{obs:3connglue}, the graph $G''$ is $H$-free. Any locally finite perturbation of $G''$ which perturbs $xy$ contains a copy of $H$ using vertices from $G'$ by Observation \ref{lem:blow_up_property}. So indeed $xy$ is fixed in $G''$.  

The fixing operation for non-edges $xy$ is analogous but relies on Corollary \ref{cor:blowup} instead of Lemma \ref{lem:blowup}.
\end{proof}

Next, we provide a technical lemma which allows us to extend the fixing operations to cores. Operations 3 and 4 will only be needed for graphs $H$ on at most 11 vertices in Section~\ref{sec:small}. We write $\delta(G)$ for the minimum degree of the graph $G$ and $\deg_G(v)$ for the degree of vertex $v\in V(G)$ in $G$.
\begin{lemma}\label{lem:handle_core}
Let $H$ and $C$ be finite graphs such that $H$  admits a fixing operation for the class of $C$-free graphs. Let $H'$ be obtained from $H$ via one of the following operations:
\begin{enumerate}
\item adding a new vertex with at most $k<\delta(C)$ neighbours, or
\item adding a new vertex with at most $\ell<\delta(\overline{C})$ non-neighbours, or
\item if $\delta(C) \geq 2$ and $C$ has no vertex $v$ with $\deg_C(v)=2$ and two non-adjacent neighbours, adding a new vertex adjacent precisely to two non-adjacent vertices of $H$,
\item if $\delta(C) \geq 2$ and $C$ has no vertex $v$ with $\deg_C(v)=2$ and two adjacent neighbours, adding a new vertex adjacent precisely to two adjacent vertices of $H$, or
\item if $C$ is 3-connected and not a clique, adding a pair of adjacent new vertices onto $H$ that are both adjacent to the same two vertices of $H$.
\end{enumerate}
Then $H'$  also admits a fixing operation for the class of $C$-free graphs. 
\end{lemma}
\begin{proof}
Throughout this proof, we start with a $C$-free graph $G_0$ and we assume that the pair $xy$ is unfixed. Moreover, we assume $H$ admits a fixing operation, so in particular there is a $C$-free graph $G$ obtained by gluing a graph onto $G_0$, such that any locally finite perturbation of $G$ which perturbs the pair $xy$ contains a copy of $H$. 

Suppose first that $H'$ is obtained from $H$ by adding a vertex $w$ to $H$ with neighbours $v_1,\dots,v_k\in H$, with $k<\delta(C)$.  
For each choice of $k$ vertices $(u_1, \dots, u_k)$ from $G$, we add an infinite independent set fully connected to the vertices $(u_1, \dots, u_k)$. Let $G'$ be the resulting graph, which is by definition obtained by gluing a graph onto $G$, and so can also be obtained by gluing a graph onto $G_0$. Crucially, all vertices in $V(G')\setminus V(G)$ have degree $k< \delta(C)$ and so none of them can be present in a copy of $C$. Since $G$ is $C$-free, $G'$ must also be $C$-free. It remains to show that the pair $xy$ is now fixed. Any locally finite perturbation in $G'$ which perturbs $xy$ must create a copy of $H$ using solely vertices of $G$. Let $u_1,\dots,u_k$ be the vertices in this copy that perform the roles of $v_1,\dots,v_k$ in $H$. In $G'$, there is an infinite independent set $S$ which is adjacent to exactly $u_1,\dots,u_k$, and for each vertex in the copy of $H$, only a finite number of edges to $S$ have been modified. Hence there is a vertex $s\in S$ for which none of the edges to this copy have been modified and this provides us with a copy of $H'$ in $G'$. Hence $H'$ admits a fixing operation for the class of $C$-free graphs.

A similar argument applies when $H'$ is obtained by adding $\ell<\delta(\overline{C})$ non-neighbours $v_1, \dots, v_\ell$ to $H$: we repeat the ``complement'' of the argument above, adding an infinite clique which is adjacent to all vertices except for $u_1,\dots,u_\ell$ (for each collection of $\ell$ distinct vertices). 

If $\delta(C)\geq 2$ and every degree 2 vertex has adjacent neighbours, then we also follow a similar construction. Let $H'$ be the graph obtained from $H$ by adding a new vertex $w$ adjacent to non-adjacent vertices $v_1,v_2$ in $H$. Let $G'$ be obtained from $G$ by adding, for each pair $uu'$ of non-adjacent vertices in $G$, an infinite independent set adjacent to $u$ and $u'$. Again, $G'$ can be obtained by gluing a graph onto $G_0$. No copy of $C$ can use the vertices from $V(G')\setminus V(G)$, since each vertex in $C$ must have degree 2, and if it has degree 2, the neighbours must be adjacent. A copy of $H'$ is created after a locally finite perturbation of the edge $xy$ in $G'$ for the same reason as above. For the fourth operation (when $\delta(C)\geq 2$ and the neighbours of degree 2 vertices are always are non-adjacent) we analogously glue an infinite independent set to pairs of adjacent vertices of $G$ instead.

Finally, suppose that $C$ is 3-connected and not a clique.
Suppose moreover that $H'$ is obtained by adding a pair $u,v$ of adjacent new vertices onto $H$ that are both adjacent to the same two vertices $a,b$ of $H$, that is, $N_H'[u]=N_H'[v]=\{a,b,u,v\}$. 

We first handle the case in which $ab\not\in E(H)$ and $C$ is a clique minus an edge, as this requires a different construction. Since $C$ is 3-connected, we must have $|V(C)|\geq 5$. For each choice of $a',b'$ from $G$ for which $a'b'\not \in E(G)$, we add two infinite independent sets to $G$ which are fully connected to each other and to $a',b'$, and for $a'b'\in E(G)$, we add an infinite clique connected to $a',b'$. 
If $C$ is any other 3-connected graph which is not a clique, then for each choice of $a',b'$ from $G$, we add an infinite clique connected to $a',b'$. Let $G'$ be the resulting graph. 

We argue again that $G'$ is $C$-free, which follows the ``gatekeeper'' idea from the introduction. 
Suppose towards a contradiction that $G'$ contains a copy $C'$ of $C$. We first argue that $C'$ must be contained in $\{a',b'\}$ union the new vertices we just glued onto $G$. 
Since $G$ is $C$-free and $C$ is non-empty, $C'$ needs to use at least one vertex $x'\in V(G')\setminus V(G)$.
Let $a',b'$ denote the pair of vertices in $G$ that $x'$ is adjacent to, and suppose that $C'$ also contains $x\in V(G)\setminus \{a',b'\}$.
Then $\{a',b'\}$ forms a 2-cut in $G'$ separating $x'$ from $x$, and hence separating $x$ from $x'$ in $C'$. 
This contradicts the assumption that $C$ is 3-connected. 
Hence, $C'\cap V(G)\subseteq \{a',b'\}$. 

When $a'b'\in E(G)$, then $C'$ is contained in a clique, a contradiction as $C$ is not a clique. When $a'b'\not \in E(G)$ and $C$ is a clique minus an edge, the graph $C'$ is contained in a tri-partite graph, which is not possible either as $C$ contains a copy of $K_4$. 
Hence, $G'$ is $C$-free as desired.

The graph $G'$ is again obtained by gluing a graph onto $G_0$ and we also repeat a similar argument to show the pair $xy$ has been fixed. Any locally finite perturbation in $G'$ which perturbs $xy$ must create a copy of $H$ using solely vertices of $G$. Let $a',b'$ be the vertices in this copy that perform the role of $a,b$ in $H$. 
Suppose first that in $G'$, there is an infinite clique $S$ which is adjacent to exactly $a',b'$. We first pick a vertex $x'\in S$ for which none of the edges to $a',b'$ have been modified. Next, we pick another vertex $y'\in S$ for which none of the edges to $a',b',x'$ have been modified. Together with $x',y'$, the copy of $H$ becomes a copy of $H'$ in the locally finite perturbation. 
The other case is when $C$ is a clique minus an edge and $a'b'\notin E(G')$, in which case we glued on two infinite independent sets fully adjacent to $a',b'$ and each other instead. We choose $x'$ from the first independent set for which none of the edges to $a',b'$ have been modified and then $y'$ from the second such that none of the edges to $x',a',b'$ have been modified.

Hence $H'$ admits a fixing operation for the class of $C$-free graphs.
\end{proof}

If the $3^*$-core $C$ of $H$ is 3-connected and not a clique, then in particular it has minimum degree at least 3. This means that we can obtain $H$ from $C$ by repeatedly adding vertices of degree at most $2<\delta(C)$ or adding pairs of adjacent vertices as in the statement of the lemma above. So starting from Lemma \ref{lem:3-conn} and then repeatedly applying Lemma \ref{lem:handle_core}, we obtain the following corollary.
\begin{corollary}\label{cor:3starcore}
Let $H$ be a finite graph with $3^*$-core (or 3-core) $H'$ which is 3-connected and not a clique. Then $H$ admits a fixing operation for the class of $H'$-free graphs. In particular, there exists a strongly $H$-induced-saturated graph.
\end{corollary}
We will also use Lemma \ref{lem:handle_core} below in Section \ref{sec:large}.

\section{A characterisation of the remaining graphs}
\label{sec:charact}
The goal of this section is to prove the following result.
\begin{theorem}
\label{thm:cases}
For any finite graph $H$ on at least $12$ vertices, either the graph $H$ or its complement $\overline{H}$ satisfies one of the following statements.
\begin{itemize}
    \item $H$ is a clique.
    \item $H$ is a forest with a unique vertex of maximum degree.
    \item The $2$-core of $H$ is $K_{2,p}$ or $K_{1,1,p}$ for some $p\geq 3$.
    \item The $3^*$-core of $H$ is 3-connected and not a clique. 
\end{itemize}
\end{theorem}
Each of these cases will be handled separately in the proof of Theorem~\ref{thm:main}.

Before we present the proof of Theorem \ref{thm:cases}, we give a few auxiliary results.
We say a graph $H$ has a \emph{butterfly cut} if we can partition its vertex set as $V(H) = A\sqcup U \sqcup B$ such that there are no edges between $A$ and $B$, $|U|\leq 2$ and $|A|,|B|\geq 3$.
\begin{observation}
\label{obs:butterfly_cut}
If $H$ has a butterfly cut, then $\overline{H}$ contains $K_{3,3}$ as a subgraph.
\end{observation}
Indeed, we can find the copy of $K_{3,3}$ with the two parts of the bipartition contained in $A$ and $B$ respectively.
\begin{lemma}\label{lem:3conRamsey}
If a finite graph $H$ has at least $12$ vertices, then either $H$ or $\overline{H}$ has a 3-connected subgraph on at least 5 vertices. 
\end{lemma}
\begin{proof}
Note that $K_{3,3}$ is a 3-connected subgraph on at least 5 vertices. So if $H$ has a butterfly cut, then we are done by Observation \ref{obs:butterfly_cut}. 
We are also done if $H$ is 3-connected.

So we may assume that there is a cut $U_1$ of size at most $2$ in $H$ and that we can partition $V(H)=A_1\sqcup U_1\sqcup B_1$ with $ 1 \leq |B_1|\leq 2$ such that there are no edges between $A_1$ and $B_1$. Since $|A_1|\geq 8$, we may assume $H[A_1]$ is not 3-connected. 
Again $H[A_1]$ must contain a cut $U_2$ of size at most 2 and we can assume it contains no butterfly cut. Hence, we can find a partition $A_1=A_2\sqcup U_2\sqcup B_2$ with $1 \leq |B_2|\leq 2$ such that there are no edges between $A_2$ and $B_2$ in $H$. Then $|A_2|\geq 4$. If $|A_2|=4$, then $|B_1|=|B_2|=2$ and $\overline{H}[A_2\cup B_1\cup B_2]$ contains the complete 3-partite graph $K_{4,2,2}$, which is 3-connected. So we may assume $|A_2|\geq 5$. We are again done if $H[A_2]$ is 3-connected, so we can split once more into $A_2= A_3\sqcup U_3\sqcup B_3$ with $|A_3|\geq 2$ and $|B_3| \geq 1$. Now $\overline{H}[A_3\cup B_1\cup B_2\cup B_3]$ contains the 3-connected subgraph $K_{2,1,1,1}$ (the complete 4-partite graph with part sizes 2, 1, 1 and 1).
\end{proof}

\begin{lemma}\label{lem:3consubgraphgood}
If a finite graph $H$ contains an induced subgraph on at least 5 vertices that is $3$-connected and not a clique, then the $3^*$-core of either $H$ or $\overline{H}$ is 3-connected and not a clique. 
\end{lemma}
\begin{proof}
Let $H$ be a finite graph with an induced subgraph $S$ on at least 5 vertices that is 3-connected and not a clique. The $3^*$-core $H_{3^*}$ of $H$ contains $S$, and is hence not a clique.
Suppose that $H_{3^*}$ is not 3-connected. Then there is a partition $V(H_{3^*})=A\sqcup U\sqcup B$ with $|U|\leq 2$ and no edges between $A$ and $B$. Since $S$ is 3-connected, we may assume that $V(S)\subseteq A\cup U$. This implies that $\overline{H}[A]$ contains at least one non-edge and that $|A|\geq 3$. We will show that the $3^*$-core of $\overline{H}$ is 3-connected.

Since $B$ is part of the $3^*$-core of $H$, it must be the case that $|B|\geq 3$. Thus $\overline{H_{3^*}}[A\cup B]$ is 3-connected and, due to the non-edge in $\overline{H}[A]$, it is not a clique. Consider the set of vertices $R=V(H)\setminus V(H_{3^*})$ that we removed from $H$ to form the $3^*$-core of $H$. We can iteratively add these back in (in clusters of size at most 2) to $\overline{H}[A\cup B]$. When a cluster is added, it has edges to all but at most two of the existing vertices (and in particular degree at least 4 since $|A|+|B|\geq 6$). Hence, $\overline{H}[A\cup B\cup R]$, which is not a clique, is 3-connected and contained in the $3^*$-core of $\overline{H}$. Since $|U|\leq 2$, the last two vertices in $U$ cannot form a separate component in the $3^*$-core of $\overline{H}$. It follows that the $3^*$-core of $\overline{H}$ is 3-connected and not a clique. 
\end{proof}

\begin{proof}[Proof of Theorem \ref{thm:cases}]
Let $H$ be a finite graph on $n\geq 12$ vertices. Let $H_0$ be the largest 3-connected subgraph of $H$. By Lemma \ref{lem:3conRamsey}, we may assume the size of $H_0$ is at least $5$ (switching from $H$ to $\overline{H}$ if necessary).
If $H_0$ is not a clique, then we are done by Lemma \ref{lem:3consubgraphgood}.
So we assume that $H_0$ is a clique. By the maximality of $H_0$, any $u\not\in H_0$ is adjacent to at most 2 vertices in $H_0$. 
We first handle the special case where $H$ consists of at most two vertices plus the clique $H_0$, in which case $|H_0| \geq 10$.
\begin{itemize}
    \item If there are no additional vertices, then $H$ is a clique, which is one of our possible conclusions.
    \item Suppose there is one additional vertex $u_1$. Note that $d(u_1)\leq 2$, since $V(H)=\{u_1\}\sqcup V(H_0)$. Thus $\overline{H}$ is a star (centred at $u_1$) with at most two isolated vertices. Since the centre of the star is a vertex of unique maximum degree, this is one of the possible conclusions.
    \item Suppose that there are two additional vertices $u_1$ and $u_2$. Both have at most two edges to $H_0$. The 2-core of $\overline{H}$ equals $K_{2,p}$ or $K_{1,1,p}$ for some $p\geq |H_0| - 4 \geq 6$, depending on whether there is an edge between $u_1$ and $u_2$ or not.
\end{itemize}
We may now assume that $|V(H)\setminus V(H_0)|\geq 3$. Let $A=V(H_0)$ and $B=V(H)\setminus V(H_0)$, and note that $|A|\geq 5$ and $|B|\geq 3$.
We will show that $\overline{H}$ has a 3-connected induced subgraph with at least 5 vertices which is not a clique so, by \Cref{lem:3consubgraphgood}, the $3^*$-core of either $H$ or $\overline{H}$ is 3-connected and not a clique.

Since each $b\in B$ has at most $2$ edges to $A$ in $H$, it has at least $|A|-2$ edges to $A$ in $\overline{H}$. This means that $\overline{H}$ contains a (not necessarily induced) bipartite subgraph $H'=(A,B,E)$ with $|A|+|B|=12$, with $|A|\geq 5$ and $|B|\geq 3$, and such that each element of $B$ has exactly $|A|-2$ neighbours in $A$.
A computer search shows that there is indeed a 3-connected subgraph on at least 5 vertices in every such bipartite graph.
Clearly, the corresponding vertices induce a 3-connected subgraph of $\overline{H}$ (as this only has more edges than the subgraph of the $H'$) and the subgraph is not a clique as it must contain at least 3 vertices of $A$ and these form an independent set in $\overline{H}$.
\end{proof}

\section{The remaining large graphs}
\label{sec:large}
We now turn our attention to the graphs $H$ satisfying the second or third property of Theorem \ref{thm:cases}, after which we will have proved Theorem~\ref{thm:main} for all graphs $H$ on at least 12 vertices.
\subsection{Fixing operations for \texorpdfstring{$K_{2,p}$}{K\textunderscore\{2,p\}}}\label{sec:K2p}
The fixing operation for $K_{2,p}$ is relatively simple.
\begin{lemma}\label{lem:K2p}
Suppose the $2$-core of $H$ is a copy of $K_{2,p}$ where $p \geq 3$. Then $H$ admits a fixing operation for the class of $K_{2,p}$-free graphs.
\end{lemma}
\begin{proof}
By Lemma \ref{lem:handle_core}, we only need to show that there is a fixing operation for $K_{2,p}$: indeed, $H$ can be obtained from $K_{2,p}$ by repeatedly adding vertices of degree at most $1<\delta(K_{2,p})$. 
    
Let $G$ be a $K_{2,p}$-free graph. Suppose first that $xy$ is an unfixed edge. To fix the edge $xy$, add an infinite independent set where every vertex is connected to both $x$ and $y$. This is equivalent to gluing on a copy of $K_{1,1,\infty}$ and any locally finite perturbation which removes the edge $xy$ gives a copy of $K_{2,p}$. It remains to check that we did not create a copy of $K_{2,p}$. Suppose that a newly added vertex $v$ is in a copy of $K_{2,p}$. Any newly added vertex is only adjacent to $x$ and $y$ and so both $x$ and $y$ must be in the copy of $K_{2,p}$. However, then the copy of $K_{2,p}$ contains a triangle, giving a contradiction.
    
Suppose next that $xy$ is an unfixed non-edge. Add a vertex $w$ which is adjacent to $x$ but not to $y$, and then add $p-1$ vertices $w_1, \dots, w_{p-1}$ adjacent to both $w$ and $y$. Finally, we blow up the (new) vertices $w, w_1, \dots, w_{p-1}$ into infinite cliques $W, W_1, \dots, W_{p-1}$. All vertices in $W$ and $y$ are adjacent to all vertices in $W_i$ for $i\in [p-1]$. It is not hard to see that any locally finite perturbation which adds the edge $xy$ creates a copy of $K_{2,p}$. In fact, there is some $w' \in W$ and $w'_i \in W_i$ for $i \in \{1, \dots, p-1\}$ such that $\{y, w'\} \cup \{x, w_1', \dots, w_{p-1}'\}$ forms a copy of $K_{2,p}$. 
    It remains to show that we have not introduced a copy of $K_{2,p}$. Let $u$ and $v$ be the two vertices of $K_{2,p}$ with $p$ neighbours, and consider which vertices could be $u$ in a copy of $K_{2,p}$. Since the $p$ neighbours of $u$ form an independent set, $u \not \in W_i$ for every $i$ as the neighbours of $w_i' \in W_i$ can be decomposed into 2 cliques and $p\geq 3$. 
    A vertex $w' \in W$ only has enough neighbours that form an independent set if $x$ is among the neighbours and there is a  vertex from each $W_i$, but then we cannot choose a suitable $v$. The only other way a copy of $K_{2,p}$ could contain one of the added vertices is for $u$ and $v$ to be $x$ and $y$, but we have not added any common neighbours of $x$ and $y$. So we did not create any copies of $K_{2,p}$ and indeed we fixed the non-edge.
\end{proof}

\subsection{Fixing operation for \texorpdfstring{$K_{1,1,p}$}{K\textunderscore\{1,1,p\}}}\label{sec:K11p}
The edge fixing operation for $K_{1,1,p}$ is more involved and requires an additional assumption about the graphs it is applied to.
\begin{lemma}\label{lem:K11p}
Suppose the $2$-core of $H$ is a copy of $K_{1,1,p}$ where $p \geq 2$. Then $H$ admits a fixing operation for the class of $K_{1,1,p}$-free graphs where the common neighbourhood of any unfixed edge can be decomposed into at most $p-2$ cliques. 
\end{lemma}
\begin{proof}

We maintain throughout that the common neighbourhood of any unfixed edge can be decomposed into at most $p-2$ cliques. For $p=2$, this means that if the edge $uv$ is unfixed, then $N(u)\cap N(v)=\emptyset$. We discuss how to define a fixing operation for $K_{1,1,p}$ then explain how to extend to graphs with this as a $2$-core.

We first handle the case $p=2$, as it is simpler and serves as a warm-up to the case $p \geq 3$. First note that any edge in an infinite clique is fixed, as perturbing it will result in an induced copy of $K_{1,1,2}$ even if this is part of an arbitrary locally finite perturbation. Given an unfixed edge $uv \in E$, we fix it by adding a countably infinite clique completely adjacent to both $u$ and $v$. Since $u$ and $v$ have no common neighbours before the fixing operation, this does not create an induced copy of $K_{1,1,2}$. Given an unfixed non-edge $uv \not\in E$, we fix it by adding a countably infinite independent set completely adjacent to both $u$ and $v$. Note that every new edge, though unfixed, is incident to a vertex of degree $2$ which is not in a triangle, so the endpoints have no common neighbour.

Assume now $p\geq 3$. We first explain how to fix an edge $uv
\in E$. We fix an infinite tree $T$ in which a particular vertex $r$ has degree 1 and all other vertices have degree $p-1$. We blow up all vertices except for $r$ into an infinite clique. We blow up $r$ into an edge and glue this onto $uv$. 

We first show that this fix keeps the graph $K_{1,1,p}$-free.
For an edge $xy$, when $x$ and $y$ correspond to different vertices of $T$, their common neighbourhood can in be decomposed into a single clique. In particular, they cannot be used as the pair of vertices of degree $p+1$ of any copy of $K_{1,1,p}$. 

When $x$ and $y$ correspond to the same vertex of $T$, their common neighbourhood can be decomposed into $p-1$ cliques; this holds for $x=u,y=v$ using our assumption that the common neighbourhood of $u,v$ could originally be covered by at most $p-2$ cliques, and for other pairs using the degree of $T$.  This shows that we did not create any copy of $K_{1,1,p}$ since no pair among the new vertices, $u$ and $v$ can be used as the vertices of degree $p + 1$. 

Next, we show that all edges $x$ and $y$ with $x,y$ corresponding to the same vertex of $T$ have been fixed (this includes the edge $uv$). 
Consider a locally finite perturbation in which $xy$ gets removed. 
Let $s_{p}$ be the vertex in $T$ that $x,y$ correspond to and let $t\neq r$ be a neighbour of $s_p$ in $T$. Let $s_1,\dots,s_{p-2}$ denote the other neighbours of $t$ in $T$.
There must be infinitely many vertices in the blow-up of $t$ that are still adjacent to $x$ and $y$ after the perturbation. We may pick any two of them, say $a$ and $b$. 
Next, we select for each $i\in [p-2]$, a vertex $w_i$ resulting from the blow-up of $s_i$, such that $\{a\}\cup\{b\}\cup\{w_1,\dots,w_{p-2},x,y\}$ forms the desired copy of $K_{1,1,p}$. Note that we can indeed do so since each of $a,b,x,y$ prohibits only finitely many options.

The new unfixed edges are the ones which correspond to different vertices of $T$, and indeed have the common neighbourhood property that we wish to maintain, since their common neighbourhood can be decomposed into a single clique (so at most $p-2$ since $p\geq 3$).

The fixing operation for a non-edge is far simpler. Indeed, let $u,v$ be an unfixed non-edge. We add an infinite independent set with neighbourhood $\{u,v\}$. This ensures that any locally finite perturbation containing the pair $u,v$ results in an induced $K_{1,1,p}$. Note that the graph resulting from the fixing operation is still $K_{1,1,p}$-free and the common neighbourhood of every edge is either unaffected by the fixing or is empty, so the fixing operation for non-edges maintains all the desired properties. 

To extend the fixing operations above to graphs for which the $2$-core is $K_{1,1,p}$ for $p\geq 2$, we unfortunately cannot directly apply Lemma~\ref{lem:handle_core} (due to the property about the common neighbourhoods). However, we may repeat the proof and note that by gluing infinite independent sets fully connected to at most one vertex, we not only maintain the property of being $K_{1,1,p}$-free, but also the property that unfixed edges have no independent set of size $p-1$ in their common neighbourhood.
\end{proof}

\subsection{Forests with a vertex of unique maximum degree}\label{sec:forest}
Let $F$ be a finite forest which has a unique vertex of maximum degree $d>1$ that we denote by $v$. One easy way of preventing a graph $G$ from containing a copy of $F$ is to ensure it has maximum degree strictly less than $d$. Unfortunately, removing an edge from $G$ is not going to increase the maximum degree and create a copy of $F$, so we need to be slightly smarter. Since $F$ is a forest, the neighbourhood of any vertex is an independent set and, instead of simply restricting the maximum degree, we can ensure that no vertex in $G$ has a neighbourhood containing an independent set of size $d$. 

We define the graph $K_\infty^p$ as the infinite graph with vertex set $\mathbb{Z}^p$ and an edge between a vertex $u = (u_1, \dots, u_p)$ and vertex $v = (v_1, \dots, v_p)$ if and only if they disagree in  exactly one coordinate i.e. there is a unique $i$ such that $u_i \neq v_i$. The maximum degree is unbounded, but the neighbourhood of any vertex can be decomposed into $p$ cliques, one for each coordinate, and the largest independent set in the neighbourhood of any vertex is of size $p$. We claim that $K_\infty^{d-1}$ is strongly $F$-induced-saturated.
\begin{lemma}\label{lem:forest}
Let $F$ be a finite forest with a unique vertex of maximum degree $d$. The graph $K_\infty^{d-1}$ is strongly $F$-induced-saturated.
\end{lemma}
\begin{proof}
As argued above, $K_\infty^{d-1}$ does not contain $F$ since the largest independent set in any neighbourhood is smaller than the maximum degree in $F$. It remains to show that making any locally finite edit creates an induced copy of $F$. 

We first consider the case in which $F$ is a tree.
Let $v$ be the unique vertex of maximum degree and suppose we remove an edge. Without loss of generality, let the edge be from $(1, 0, \dots, 0)$ to $(2, 0, \dots, 0)$. 
There must exist an integer $i_v$ such that $(i_v,0,\dots,0)$ is still connected to both $(1,0,\dots,0)$ and $(2,0,\dots,0)$. Here we will embed our vertex $v$. Similarly, there are
integers $j_3, \dots,j_d$ such that 
after the locally finite edit, there are still edges from $(i_v,0,\dots,0)$ to 
\[
(i_v,j_3,0,\dots,0),(i_v,0,j_4,\dots,0),\dots,(i_v,0,0,\dots,j_d)
\]
and the vertices displayed above together with $(1,0,\dots,0),(2,0,\dots,0)$ form an independent set. We embed the $d$ neighbours of $v$ into this independent set of size $d$.

We now explore the tree, heading out from $v$ and assigning each vertex to a vertex in $\mathbb{Z}^{d-1}$ as we go. Suppose we have reached the vertex $u = (u_1, \dots, u_{d-1})$ which disagrees with the previous vertex in coordinate $i$. There are at most $d' \leq d-1$ neighbours of $u$, say $w_0, w_1, \dots, w_{d'-1}$, and we have already seen one of them, say $w_0$. We must assign to each $w_j$ with $j\in [d]$ a vertex in $\mathbb{Z}^{d-1}$. First choose $d'-1$  distinct coordinates $p_1, \dots, p_{d'-1}$, none of which are equal to $i$, and set $w_j$ to be equal to $u$ except in position $p_j$. For position $p_j$, we choose an integer which ensures that $w_j$ is not adjacent to any of the vertices we have embedded except for $u$. Since at each step, there are only finitely many ``bad integers'' to choose out of infinitely many, we can always select an appropriate one. 

Adding an edge can be handled similarly: we embed $v$ into an endpoint $x$ of the edge and note that the added edge again makes it possible to embed the $d$ neighbours of $v$ as an independent set in the neighbourhood of $x$.

For arbitrary forests $F$, let $T$ be the connected component of $F$ which contains the unique vertex of maximum degree. We showed above that any locally finite edit creates a copy of $T$. All other connected components of $F$ have maximum degree at most $d-1$. After embedding $T$, it is straightforward to embed the other components in $K_\infty^{d-1}$ in a way that there are no edges to the vertices used for $T$.
\end{proof}

\subsection{Scheduling the fixes}\label{subsec:scheduling}
The following lemma is straightforward but important. A key part of the proof is that any bad pair is fixed in a finite number of steps. This ensures every bad pair is fixed in the limit. 
\begin{lemma}
\label{lem:fix_to_strong}
Suppose that a finite graph $H$ admits a fixing operation for a non-empty class of $H$-free graphs. Then there exists a strongly $H$-induced-saturated graph.
\end{lemma}
\begin{proof}
We define a sequence of graphs $G_1\subseteq G_2 \subseteq \dots$ in $\mathcal{F}$ and we define the graph $G$ as $\cup_{i=1}^\infty G_i$.

Let $G_1\in \mathcal{F}$. We keep track of a set $F$ of pairs that need to be fixed with a priority order (described later in this proof), which is initialised as the (countable) set of unfixed pairs in $G_1$. 
Suppose we have defined $G_i\in \mathcal{F}$ for some integer  $i\geq1$. If all pairs $xy\in G_i$ are fixed, we may set $G_t=G_i$ for all $t\geq i$ and $G=G_i$ is our desired graph.
Otherwise, let $xy\in G_i$ be an unfixed pair of highest priority in $F$. We let $G_{i+1}$ be the graph obtained from applying the fixing operation for this pair. 

Next, we describe how to define the priority order in such way that every bad pair is eventually fixed. We first fix an enumeration $<$ on $\mathbb{N}\times \mathbb{N}$, e.g. 
\[
(1,1)<(1,2)<(2,1)<(2,2)<(3,1)<(3,2)<(3,3)<\dots
\]
At any point of our sequence, $G_i$ is a countable graph and when created, we may enumerate the vertex pairs of $G_i$. We prioritise the $j$th vertex pair of $G_i$ above the $j'$th vertex pair of $G_{i'}$ if and only if $(i,j)<(i',j')$. At any point, we perform next the fix of highest priority among those that involve a pair of vertices both existing in the current graph and for which the vertex pair is currently unfixed. (If desired, we can also allow to fix vertex pairs that are already fixed.)

By definition of a fixing operation for $xy$, if $G_i\in \mathcal{F}$, then $G_{i+1}\in \mathcal{F}$ and the pair $xy$ is fixed in $G_{i+1}$. Moreover, any pair which was fixed in $G_i$ remains fixed in $G_{i+1}$ since $G_i$ is an induced subgraph of $G_{i+1}$.
By putting the priority order on $F$, we ensure that for any pair $xy$ in $G$, there is a finite $i$ such that $xy$ is a fixed pair in $G_i$, so in particular $xy$ is a fixed pair in $G$. 

Since no copy of $H$ can be created in $G_i$ for all $i$, $G$ is also $H$-free. This proves that $G$ has the desired properties.
\end{proof}

\section{The remaining small graphs}
\label{sec:small}
In this section we consider graphs on at most 11 vertices. While many of these graphs can be handled using the lemmas and theorems above, there are still several small graphs which we still need to consider.
To this end, we introduce two new constructions to handle families of graphs (in \Cref{sec:upandright,sec:torero}), and a method for checking for fixing operations with the help of a computer (in \Cref{sec:computermagic}). This still leaves a total of eight graphs, four graphs and their complements, which we handle individually with two more constructions at the end of this section.

Before we give the new constructions and the method to find fixing operations, let us briefly summarise the steps in the computer search. For each graph $G$ on at most 11 vertices, we check if any of the following hold.
\begin{enumerate}
    \item The graph $H$ is a non-empty forest with a unique vertex of maximum degree (\Cref{lem:forest}).
    \item The 2-core of $H$ is a copy $K_{2,p}$ with $p \geq 3$ (\Cref{lem:K2p}).
    \item The 2-core of $H$ is a copy of $K_{1,1,p}$ where $p \geq 2$ (\Cref{lem:K11p}).
    \item The $3$-core of $H$ is 3-connected and not a clique (\Cref{cor:3starcore}).
    \item The $(1,1)$-core of $H$ is a copy of $P_4$ or the bull graph (\Cref{thm:bull}).
    \item The graph $H$ is ``close to'' a permutation graph (\Cref{thm:upandright}).
    \item \label{it:7} The $2$-core, $3$-core, $2$-edge-core or $2$-non-edge-core have fixing operations (see \Cref{sec:computermagic}).
\end{enumerate}
    If none of these hold for $H$, we check if any of them hold for its complement $\overline{H}$. As mentioned earlier, these checks leave just eight graphs, \texttt{E?qw}, \texttt{F?S|w}, \texttt{F?q|w} and \texttt{F?q~w}, and their complements, which we resolve in \Cref{sec:circ,sec:z3}. 

    The code used for this computer check is attached to the arXiv submission. 

\subsection{The up and right graph}
\label{sec:upandright}
We start with a construction which we call the \emph{up and right graph}, which handles all graphs which are ``close to" a permutation graph. A graph $G$ on the vertices $\{v_1, \dots,v_n\}$ is a \emph{permutation graph} if there is a permutation $\sigma \in S_n$ such that, for every $i < j$, the edge $v_iv_j$ is present if and only if $\sigma(i) < \sigma(j)$. A graph is \emph{close to} a permutation graph if it is not a permutation graph, but adding or removing any edge creates a permutation graph. 

The \emph{up and right graph} is the graph on $\mathbb{Q} \times \mathbb{Q}$ where the vertex $(p,q)$  is connected to  $(s,t)$ if and only if one of the following holds

\begin{itemize}
    \item $(q + r \sqrt2) < (s + t \sqrt2)$ and $(q - r \sqrt2) < (s - t \sqrt2)$
    \item $(q + r \sqrt2) > (s + t \sqrt2)$ and $(q - r \sqrt2) > (s - t \sqrt2)$
\end{itemize}

That is, two vertices are connected if one of the vertices is up and right of the other after the linear transformation $(q,r) \mapsto (q + r \sqrt2, q - r \sqrt2)$ has been applied. Our motivation for this is to ensure that no two vertices can differ in exactly one coordinate. Indeed, if \[q \pm r \sqrt2 = s \pm t \sqrt2 \] where $q,r,s,t \in \mathbb{Q}$, then we must have $(q,r) = (s,t)$. This means that given $n$ distinct vertices $v_i = (q_i, r_i)$ for $i \in [n]$, we can order them such that \[q_i + r_i \sqrt2 < q_j + r_j \sqrt2 \] whenever $i < j$.

Similarly, there is also an order based on the other coordinate, and there is a unique permutation $\sigma \in S_n$ such that \[ q_{i} - r_{i} \sqrt2 < q_{j} - r_{j} \sqrt2 \]
whenever $\sigma(i) < \sigma(j)$. With this notation, there is an edge between $v_i$ and $v_j$, where $i < j$, exactly when $\sigma(i) < \sigma(j)$. This means that any induced subgraph is a permutation graph.

The following lemma shows that every permutation graph can be found as an induced subgraph of the up and right graph, giving us a characterisation of exactly which finite graphs appear as induced subgraphs.

\begin{lemma}\label{lem:perm}
A graph $H$ is an induced subgraph of the up and right graph if and only if it is a permutation graph.
\end{lemma}
\begin{proof}
We already argued that any induced subgraph is a permutation graph, so it only remains to show how to find a copy of any given permutation graph. Let $\sigma$ be a permutation that defines the permutation graph $H$ and consider the $n$ points \[\tilde{v}_i = \left( \frac{i + \sigma(i)}{2} , \frac{i - \sigma(i)}{2 \sqrt{2}} \right) \text{ for }i\in [n].\] 
The coordinates are chosen so that
\[
\frac{i + \sigma(i)}{2}+\frac{i - \sigma(i)}{2 \sqrt{2}} \sqrt{2} = i
\]
and 
\[
\frac{i + \sigma(i)}{2}-\frac{i - \sigma(i)}{2 \sqrt{2}} \sqrt{2}=\sigma(i).
\]
Ideally, we would take the $\tilde{v}_i$ as the vertices in the up and right graph, but the $\tilde{v}_i$ are not necessarily rational. However, the rational points are dense in $\mathbb{R}^2$ and we can choose vertices of the up and right graph which are sufficiently close to the points $\tilde{v}_i$ to keep the expected adjacencies. This provides a subgraph isomorphic to~$H$.
\end{proof}

Given this lemma is it not hard to prove the following theorem classifying the graphs which are strongly saturating in the up and right graph.

\begin{theorem}\label{thm:upandright}
The up and right graph is strongly $H$-induced-saturated if and only if the following hold.
\begin{enumerate}
    \item $H$ is not a permutation graph.
    \item There exists an edge $e$ such that $H-e$ is a permutation graph.
    \item There exists a non-edge $\overline{e}$ such that $H + \overline{e}$ is a permutation graph.
\end{enumerate}
\end{theorem}
\begin{proof}
It is easy to see these conditions are necessary. Indeed, the up and right graph cannot contain a copy of $H$ so, by Lemma \ref{lem:perm}, $H$ is a not a permutation graph. Adding an edge $e$ to the up and right graph must create a copy of $H$ and $e$ corresponds to some edge $e'$ of this copy of $H$. This means there is an induced copy of $H - e'$ in the up and right graph and $H - e'$ is a permutation graph. Similarly, there must be a non-edge $\overline{e}$ of $H$ such that $H + \overline{e}$ is a permutation graph.

We now show that the conditions are sufficient. Since $H$ is not a permutation graph there is no copy of $H$ in the up and right graph, and it remains to show that making any locally finite edit creates a copy of $H$. Let us first consider adding a single edge. By assumption there is an edge $e = v_i v_j$ of $H$ such that $H-e$ is a permutation graph, and we claim that this is enough to guarantee the existence of a copy of $H$ after adding an edge.

Suppose $H-e$ is a permutation graph and let $f$ be a non-edge from $(q,r)$ to $(s,t)$ in the up and right graph.
The up and right graph is vertex transitive so we can assume that $(q,r) = (0,0)$. We now repeat the construction from Lemma \ref{lem:perm}, but with the addition of extra scaling factors. Define

\[ k' = (k -i) \frac{s + t\sqrt2}{j - i} \text{ and } \sigma'(k) = (\sigma(k) - \sigma(i)) \frac{s - t\sqrt2}{\sigma(j) - \sigma(i)},  \]
and let 
\[ \tilde{v}_k = \left( \frac{k' + \sigma'(k)}{2}, \frac{k' - \sigma'(k)}{2 \sqrt{2}}\right).\] Note that by our choice of the scaling factors, we have $\tilde{v}_i = (0,0)$ and $\tilde{v}_j = (s,t)$, so that $e$ corresponds to $f$. Since $e$ is not an edge in the permutation graph, $j - i$ and $\sigma(j) - \sigma(i)$ have opposite signs. Likewise, since $f$ is not an edge in the up and right graph, ${s + t\sqrt2}$  and ${s - t\sqrt{2}}$ also have opposite signs, and the scaling factors both have the same sign. Hence, if the $\tilde{v}_k$ were vertices in the up and right graph, we would have found an induced subgraph isomorphic to $H - e$ and we would be done. Instead, we may need to perturb the $\tilde{v}_k$ by a small amount for $k \neq i, j$ to make sure that they are rational.

Using an almost identical argument one can show that removing an edge creates a copy of $H$ provided that there is a non-edge $\overline{e}$ of $H$ such that $H + \overline{e}$ is a permutation graph.

So far we have only shown that the graph is induced-saturated for $H$ and contains a copy of $H$ when a single perturbation is made, but it is not hard to adapt the proof to handle locally finite perturbations. We begin as before, setting $v_i$ and $v_j$ as $\tilde{v}_i$ and $\tilde{v}_j$ respectively. We then choose each $v_k$ in turn. At each step, there are infinitely many choices for $v_k$ that are sufficiently close to $\tilde{v}_k$. Since the edit is locally finite and we have only picked finitely many vertices so far, we can choose such a vertex (in fact, all but finitely many work) for which none of the adjacencies to the vertices we have already chosen have been altered. 
\end{proof}

\subsection{The torero graph}
\label{sec:torero}

In this section, we give a construction which is strongly-$H$-induced-saturated whenever the $(1,1)$-core is a copy of the bull graph (see \Cref{fig:bull}) or $P_4$. 

\begin{figure}[ht!]
    \centering
    \begin{tikzpicture}
        \draw (-0.5, 0.866) node(0)[style=vertex]{};
        \draw (-0.5, 0.866) node[above]{$x$};
        \draw (0.5, 0.866) node(1)[style=vertex]{};
        \draw (0.5, 0.866) node[above]{$y$};
        \draw (0, 0) node(2)[style=vertex]{};
        \draw (0, 0) node[below]{$z$};
        \draw (-0.5,0.866) ++ (160:1) node(3)[style=vertex]{};
        \draw (-0.5,0.866) ++ (160:1) node[left]{$a$};
        \draw (0.5, 0.866) ++ (20:1) node(4)[style=vertex]{};
        \draw (0.5, 0.866) ++ (20:1) node[right]{$b$};
        
        \draw (0) -- (1);
        \draw (0) -- (2);
        \draw (1) -- (2);
        \draw (0) -- (3);
        \draw (1) -- (4);
    \end{tikzpicture}
    \caption{The bull graph.}
    \label{fig:bull}
\end{figure}
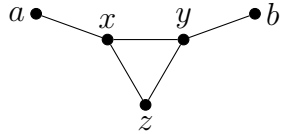

The \emph{torero graph} has vertex set $\mathbb{Q} \cap (0,1)$ and there is an edge from $x$ to $y$ if and only if $x + y  > 1$. 
\begin{theorem}\label{thm:bull}
The torero graph is strongly $H$-induced-saturated whenever the $(1,1)$-core of $H$ is a copy of the bull graph or $P_4$.
\end{theorem}
\begin{proof}
We prove this theorem in three parts. First, we show that there is no induced copy of $P_4$ in the torero graph (and hence no copy of the bull graph). Then we show that any locally finite modification of the torero graph contains an induced copy of the bull graph (and hence an induced copy of $P_4$). Finally, we observe that the torero graph is strongly-$H$-induced-saturated when the $(1,1)$-core is a copy of the bull graph or $P_4$.

Label the vertices of $P_4$ by $v_1, \dots, v_4$ so that the edges are $v_1v_2$, $v_2v_3$ and $v_3v_4$. Since $v_1v_2$ is an edge and $v_1v_3$ is not an edge, $v_1 + v_2 > 1 > v_1 + v_3$. In particular, $v_2 > v_3$. However,  since $v_3v_4$ is an edge and $v_2v_4$ is not an edge, we have $v_2 < v_3$, a contradiction.

Suppose that we remove an edge $bx$ from the torero graph where $b < x$. Since this is normally an edge, we have that $b + x >  1$. There is therefore some $y \in \mathbb{Q}$ such that  $b < y < x$ and $b + y > 1$. Now choose $a$ and $z$ such that $a + y < 1 < a + x $ and $z + b < 1 < z + y$, which is possible since $b < y < x$. This gives an induced copy of the bull when we remove a single edge from the torero graph. Since at each step we are choosing any rational in an interval, of which there are infinitely many, we can always choose a vertex such that no edge/non-edge to any of the preceding vertices has been perturbed. This means that we can also manage all locally finite edits.
Adding an edge follows a similar argument where the new edge forms the edge $by$ in the bull. Diagrams of both cases can be seen in Figure \ref{fig:bull-constructions}.

\begin{figure}
    \centering
    \begin{subfigure}{0.45 \textwidth}
        \centering
        \begin{tikzpicture}[xscale=6]
            \draw (0.85, 0.5) node(0)[style=vertex]{};
            \draw (0.85, 1) node[](l0){$x$};
            \draw (0.6, -0.5) node(1)[style=vertex]{};
            \draw (0.6,1) node[](l1){$y$};
            \draw (0.45, 0) node(2)[style=vertex]{};
            \draw (0.45, 1) node[](l2){$z$};
            \draw (0.2,0.4) node(3)[style=vertex]{};
            \draw (0.2, 1) node[](l3){$a$};
            \draw (0.35, -0.5) node(4)[style=vertex]{};
            \draw (0.35, 1) node[](l4){$b$};
        
            \draw (0) -- (1);
            \draw (0) -- (2);
            \draw (1) -- (2);
            \draw (0) -- (3);
            \draw (1) -- (4);
            \draw[dashed] (0) -- (4);
            
            \draw[style=help lines, dashed] (0.85,-1) -- (l0);
            \draw[style=help lines, dashed] (0.6,-1) -- (l1);
            \draw[style=help lines, dashed] (0.45,-1) -- (l2);
            \draw[style=help lines, dashed] (0.2,-1) -- (l3);
            \draw[style=help lines, dashed] (0.35,-1) -- (l4);
            
            \draw (0, -1) node[below]{0} -- (0.25, -1) node[below]{}  -- (0.5, -1) node[below]{} -- (0.75, -1) node[below]{} -- (1, -1) node[below]{$1$};
            \draw[style=help lines] (0,-1) -- (0,1);
            \draw[style=help lines] (1,-1) -- (1,1);
        \end{tikzpicture}
        \caption{}
    \end{subfigure}
    \begin{subfigure}{0.45 \textwidth}
        \centering
        \begin{tikzpicture}[xscale=6]
            \draw (0.85, 0.5) node(0)[style=vertex]{};
            \draw (0.85, 1) node[](l0){$x$};
            \draw (0.3, -0.5) node(1)[style=vertex]{};
            \draw (0.3,1) node[](l1){$y$};
            \draw (0.75, 0) node(2)[style=vertex]{};
            \draw (0.75, 1) node[](l2){$z$};
            \draw (0.2,0.4) node(3)[style=vertex]{};
            \draw (0.2, 1) node[](l3){$a$};
            \draw (0.1, -0.5) node(4)[style=vertex]{};
            \draw (0.1, 1) node[](l4){$b$};
        
            \draw (0) -- (1);
            \draw (0) -- (2);
            \draw (1) -- (2);
            \draw (0) -- (3);
            \draw[dashed] (1) -- (4);
            
            \draw[style=help lines, dashed] (0.85,-1) -- (l0);
            \draw[style=help lines, dashed] (0.3,-1) -- (l1);
            \draw[style=help lines, dashed] (0.75,-1) -- (l2);
            \draw[style=help lines, dashed] (0.2,-1) -- (l3);
            \draw[style=help lines, dashed] (0.1,-1) -- (l4);
            
            \draw (0, -1) node[below]{0} -- (0.25, -1) node[below]{}  -- (0.5, -1) node[below]{} -- (0.75, -1) node[below]{} -- (1, -1) node[below]{$1$};
            \draw[style=help lines] (0,-1) -- (0,1);
            \draw[style=help lines] (1,-1) -- (1,1);
        \end{tikzpicture}
        \caption{}
    \end{subfigure}
    \caption{We show in (a) how to create a bull when the edge $bx$ is removed and  we show in (b) how to create a bull when $by$ is added.}
    \label{fig:bull-constructions}
\end{figure}
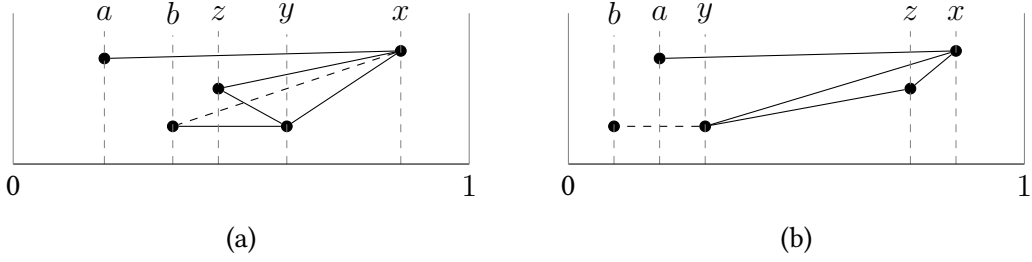

Finally, we explain how to handle any graph $H$ whose (1,1)-core is a copy of the bull graph or $P_4$. Since the torero graph has no induced copy of $P_4$, it also has no copy of $H$. To embed a copy of $H$ after a locally finite perturbation, we first embed a copy of the bull graph or $P_4$ as appropriate. The graph $H$ may be obtained from either the bull or $P_4$ by iteratively adding vertices connected to none of the preceding vertices, or to all of them. Given any set of vertices, there are infinitely many vertices which are connected to every vertex in the set (those sufficiently close to 1) and infinitely many vertices which are connected to none of the vertices (those sufficiently close to 0), and we can find a copy of $H$ by iteratively choosing such vertices.
\end{proof}

\subsection{Computer check for gatekeepers}\label{sec:computermagic}
Let us say that a 2-cut $\{u,v\}$ of $G$ is an \emph{edge-2-cut} if $\{u,v\}$ is an edge in $G$, and otherwise we say it is a \emph{non-edge-2-cut}.

Let $H$ be a 2-connected graph and let $e=xy$ be an edge of $H$. We show that if there is no non-edge-2-cut $\{u,v\}$ of $H$ such that a component of $H - \{u,v\}$ is an induced subgraph of $H-e$, then $e$ is a gatekeeper. 
Suppose towards a contradiction that the edge $e = xy$ is not a gatekeeper. Then there is some graph $G$ which does not contain a copy of $H$ such that gluing $H - e$ to $G$ by identifying $x$ and $y$ with the endpoints of a non-edge of $G$ creates a copy $H'$ of $H$. The vertices $x$ and $y$ are a $2$-cut of $H'$ 
and removing them splits $H'$ into components, say, $C_1, \dots, C_r$.  At least one of these components is disjoint from $G$, else $H'$ would be contained entirely in $G$. Let $u$ and $v$ be the vertices from $H$ that correspond to $x$ and $y$ in $H'$. Then the component contained in $H' -\{x,y\}$ corresponds to a  component of $H-\{u,v\}$ which is an induced subgraph of $H-e$. 

In fact, we can strengthen this condition slightly by including how a component of $H - \{u,v\}$ connects to the vertices $\{u,v\}$, and this is what we use in practice. That is, for each non-edge-2-cut $\{u,v\}$ which splits $H$ into components $C_1, \dots, C_r$ (say), we check if there is a copy of $H[V(C_i) \cup \{u,v\}]$ in $H-e$ where the vertices $\{u,v\}$ correspond to $\{x,y\}$ (either way round).
An example is given in \Cref{fig:cut-criterion}.

For the constructed graph to be \emph{strongly}-$H$-induced-saturated it is not sufficient to simply glue on a copy of $H-e$, even when $e$ is a gatekeeper. Instead, we will replace each vertex in $H-e$ except for $x$ and $y$ by either infinite cliques or independent sets. As seen in the proof of \Cref{lem:blowup}, this can be done if neither vertex is a true twin or if neither vertex is a false twin, and we check for this condition before we check if a particular edge or non-edge is a gatekeeper using the above method.

Let us summarise by describing the implementation.
First, the code checks that the graph is 2-connected and not a complete graph. Then all of the 2-cuts are enumerated and split into edge-2-cuts and non-edge-2-cuts. 
For each edge-2-cut $e=xy$, we check that $x$ and $y$ do not have twins of differing types, and we move onto the next edge-2-cut if they do. 
We then replace $e$ by a ``red'' edge, which serves to colour the two vertices $x$ and $y$.
We loop over the non-edge-2-cuts and, for each non-edge-2-cut $\{u,v\}$, we find the connected components of $H - \{u,v\}$.
For each connected component, $C$, we form the subgraph $H[V(C) \cup \{u,v\}]$ and add in a red edge between $u$ and $v$.
If this subgraph is isomorphic (including edge colours) to a subgraph of $H$, then we move onto another edge-2-cut.
If we did not find such a subgraph for any choice of non-edge-2-cut and connected component, then we have found the suitable fixing operation and we move onto non-edge-2-cuts, which are handled similarly.

This gives us a method to look for fixing operations in a given $2$-connected graph, but we can combine this with \Cref{lem:handle_core} to cover many more graphs. Indeed, suppose that the $k$-core of $H$, which we denote $H'$, has a fixing operation (potentially found by the method above) for the class of $H'$-free graphs. Then, by repeatedly applying \Cref{lem:handle_core}, the graph $H$ also admits a fixing operation for the class of $H'$-free graphs and, in particular, is strongly saturating.
Using the third and fourth parts of \Cref{lem:handle_core}, we can also generalise this to what we call the 2-edge-core and the 2-non-edge-core. 
The \emph{2-edge-core} (resp. \emph{2-non-edge-core}) of a graph is formed by repeatedly removing vertices of degree less than $2$ and vertices of degree 2 whose neighbours are adjacent (resp. non-adjacent). The 2-edge-core of a graph has minimum degree at least 2 and no vertex $v$ with degree $2$ whose neighbours are adjacent. Using \Cref{lem:handle_core} repeatedly, it follows that the graph $H$ admits a fixing operation if its 2-edge-core or 2-non-edge-core admits a fixing operation, as required for check \ref{it:7} above.

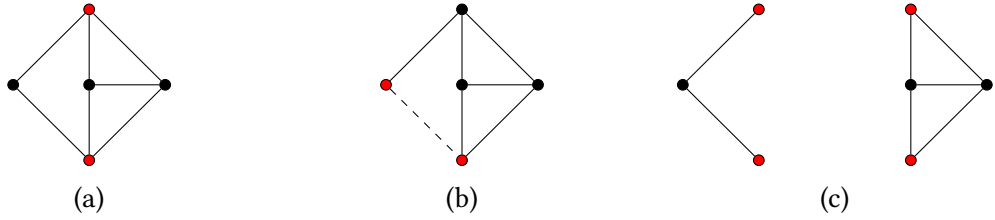
\begin{figure}
    \centering
    \begin{subfigure}{0.3\textwidth}
    \centering
        \begin{tikzpicture}
             \draw (-1, 0) node(0)[style=vertex]{};
             \draw (0, -1) node(1)[style=vertex, fill=red]{};
             \draw (0, 1) node(2)[style=vertex, fill=red]{};
             \draw (0, 0) node(3)[style=vertex]{};
             \draw (1, 0) node(4)[style=vertex]{};

             \draw (1) -- (3) -- (2) -- (0) -- (1) -- (4) -- (2);
             \draw (3) -- (4);
        \end{tikzpicture}
    \caption{}
    \end{subfigure}
            \begin{subfigure}{0.3\textwidth}
        \centering
        \begin{tikzpicture}
             \draw (-1, 0) node(0)[style=vertex, fill=red]{};
             \draw (0, -1) node(1)[style=vertex, fill=red]{};
             \draw (0, 1) node(2)[style=vertex]{};
             \draw (0, 0) node(3)[style=vertex]{};
             \draw (1, 0) node(4)[style=vertex]{};

            \draw (0) -- (2) -- (4) -- (1) -- (3) -- (2);
            \draw (3) -- (4);
            \draw[dashed] (0) -- (1);
        \end{tikzpicture}
    \caption{}
    \end{subfigure}
    \begin{subfigure}{0.3\textwidth}
    \centering
        \begin{tikzpicture}
             \draw (-1, 0) node(0a)[style=vertex]{};
             \draw (0, -1) node(1a)[style=vertex, fill=red]{};
             \draw (0, 1) node(2a)[style=vertex, fill=red]{};

             \draw (2a) -- (0a) -- (1a);

             \begin{scope}[shift={(2,0)}]
                \draw (0, -1) node(1b)[style=vertex, fill=red]{};
                \draw (0, 1) node(2b)[style=vertex, fill=red]{};
                \draw (0, 0) node(3b)[style=vertex]{};
                \draw (1, 0) node(4b)[style=vertex]{};
                \draw (3b) -- (2b) -- (4b) -- (1b) -- (3b) -- (4b);
             \end{scope}
        \end{tikzpicture}
    \caption{}
    \end{subfigure}

    \caption{
    (a)~The graph \texttt{Dr[} with its only 2-cut highlighted in red. There are no 2-cuts which are edges, so every non-edge is a gatekeeper.
    (b)~The graph with the edge $xy$ removed, with $x$ and $y$ highlighted in red. 
    (c)~The two fragments of the graph created by taking the components of the only 2-cut and adding back in the 2-cut. Observe that there is no copy of either of the fragments in (b) with matching vertex colours, which means that $xy$ is a gatekeeper.    }
    \label{fig:cut-criterion}
\end{figure}

\subsection{The rational geometric graph}\label{sec:circ}
The \emph{rational geometric graph} is the graph on $\mathbb{Q}$ where there is an edge between $q$ and $r$ if and only if $|q - r| < \pi$. Our main interest in this construction is to handle the three problematic small graphs shown in Figure \ref{fig:circulant_graphs}. First, we make the easy observation that the neighbourhood of any vertex $v$ in the rational geometric graph can be partitioned into two cliques (where there may be edges between the cliques). In particular, this immediately shows that none of the graphs in Figure \ref{fig:circulant_graphs} are induced subgraphs of the rational geometric graph.

\begin{figure}[ht!]
    \centering
            \begin{subfigure}{0.3\textwidth}
        \centering
        \begin{tikzpicture}
            \draw (70:0.5) node(a)[style=vertex]{};
            \draw (110:0.5) node(b)[style=vertex]{};
            \draw (0,0) node(c)[style=vertex, fill=red]{};
            \draw (-0.5,-0.866) node(d)[style=vertex]{};
            \draw (0.5, -0.866) node(e)[style=vertex]{};
            \draw (1, -0.866) node(f)[style=vertex]{};
            
            \draw (a)  -- (c);
            \draw (b)  -- (c);
            \draw (c)  -- (d);
            \draw (c)  -- (e);
            \draw (d)  -- (e);
            \draw (e)  -- (f);
        \end{tikzpicture}
        \caption{\texttt{E?qw}}
\end{subfigure}
    \begin{subfigure}{0.3\textwidth}
    \centering
    \begin{tikzpicture}
        \draw (0,0) node(a)[style=vertex, fill=red]{};
        \draw (0,1) node(b)[style=vertex]{};
        \draw (1,0) node(c)[style=vertex]{};
        \draw (1,1) node(d)[style=vertex]{};
        \draw (0.5, 1.866) node(e)[style=vertex]{};
        \draw (200:0.5) node(f)[style=vertex]{};
        \draw (250:0.5) node(g)[style=vertex]{};
        
        \draw (a)  -- (c);
        \draw (a)  -- (b);
        \draw (a)  -- (d);
        \draw (a)  -- (c);
        \draw (a)  -- (f);
        \draw (a)  -- (g);
        \draw (c)  -- (d);
        \draw (b)  -- (d);
        \draw (b)  -- (c);
        \draw (b)  -- (e);
        \draw (d)  -- (e);
    \end{tikzpicture}
    \caption{\texttt{F?rLw}}
    \end{subfigure}
    \begin{subfigure}{0.3\textwidth}
    \centering
        \begin{tikzpicture}
            \draw (0,0) node(a)[style=vertex]{};
            \draw (0,1) node(b)[style=vertex, fill=red]{};
            \draw (1,0) node(c)[style=vertex]{};
            \draw (1,1) node(d)[style=vertex]{};
            \draw (0.5, 1.866) node(e)[style=vertex]{};
            \draw (200:0.5) node(f)[style=vertex]{};
            \draw (0,1) ++ (200:0.5) node(g)[style=vertex]{};
            
            \draw (a)  -- (c);
            \draw (a)  -- (b);
            \draw (a)  -- (d);
            \draw (a)  -- (c);
            \draw (a)  -- (f);
            \draw (b)  -- (g);
            \draw (c)  -- (d);
            \draw (b)  -- (d);
            \draw (b)  -- (c);
            \draw (b)  -- (e);
            \draw (d)  -- (e);
        \end{tikzpicture}
        \caption{\texttt{F?S|w}}
        \end{subfigure}

\caption{The three problematic graphs which we handle using the rational geometric graph. The vertices highlighted in red have neighbourhoods which cannot be split into two cliques.}
\label{fig:circulant_graphs}
\end{figure}
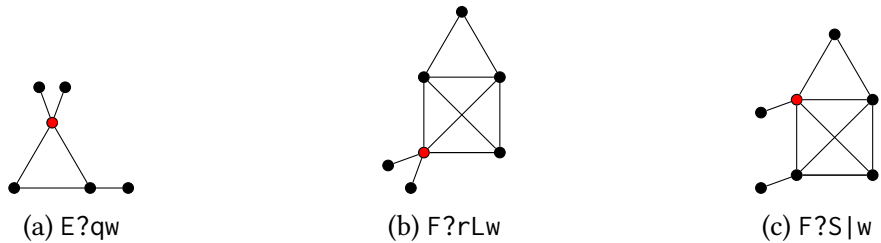

To handle these graphs, what remains to show is that making any locally finite edit to the rational geometric graph creates an induced copy. Let us first consider making just a single change, either swapping an edge to a non-edge or a non-edge to an edge. By the symmetry of the rational geometric graph, we can assume the edge/non-edge is from $0$ to some $r > 0$, where $r < \pi$ if it is an edge and $r > \pi$ if it is a non-edge. It is straightforward to find a sequence of intervals from which we can choose the vertices of our small graphs, but in the interest of conciseness we only sketch the constructions in Figure \ref{fig:circulant_2}. 

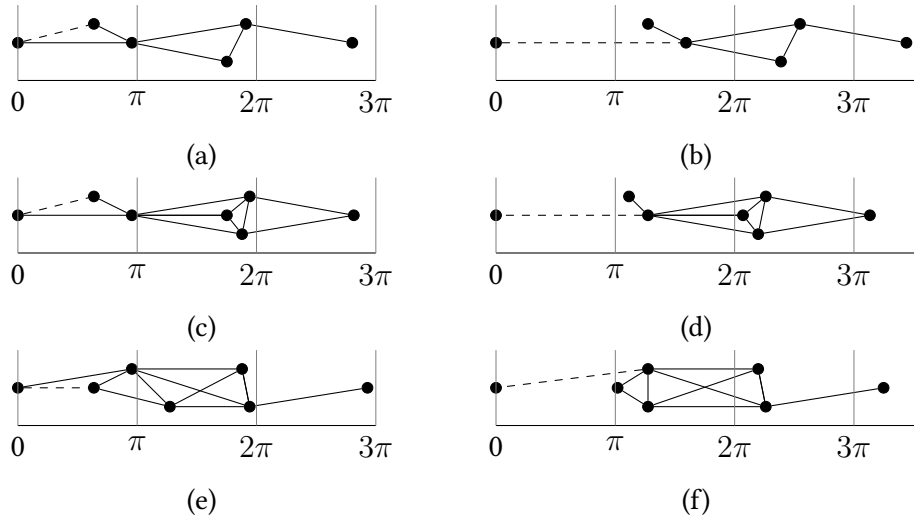
\begin{figure}[ht!]
    \centering
        \begin{subfigure}{0.4\textwidth}
        \centering
        \begin{tikzpicture}[scale=0.5]
            \draw (0,0) node(a)[style=vertex]{};
            \draw (2,0.5) node(b)[style=vertex]{};
            \draw (3,0) node(c)[style=vertex]{};
            \draw (5.5,-0.5) node(d)[style=vertex]{};
            \draw (6, 0.5) node(e)[style=vertex]{};
            \draw (8.8, 0) node(f)[style=vertex]{};
            
            \draw[dashed] (a) -- (b);
            \draw (a)  -- (c);
            \draw (b)  -- (c);
            \draw (c)  -- (d);
            \draw (c)  -- (e);
            \draw (d)  -- (e);
            \draw (e)  -- (f);
            \draw (0, -1) node[below]{0} -- (3.14, -1) node[below]{$\pi$}  -- (6.28, -1) node[below]{$2\pi$} -- (9.42, -1) node[below]{$3\pi$};
            \draw[style=help lines] (0,-1) -- (0,1);
            \draw[style=help lines] (3.14,-1) -- (3.14,1);
            \draw[style=help lines] (6.28,-1) -- (6.28,1);
            \draw[style=help lines] (9.42,-1) -- (9.42,1);
        \end{tikzpicture}
        \caption{} 
    \end{subfigure}
    \begin{subfigure}{0.4\textwidth}
        \centering
        \begin{tikzpicture}[scale=0.5]
            \draw (0,0) node(a)[style=vertex]{};
            \draw (4,0.5) node(b)[style=vertex]{};
            \draw (5,0) node(c)[style=vertex]{};
            \draw (7.5,-0.5) node(d)[style=vertex]{};
            \draw (8, 0.5) node(e)[style=vertex]{};
            \draw (10.8, 0) node(f)[style=vertex]{};
            
            \draw[dashed] (a)  -- (c);
            \draw (b)  -- (c);
            \draw (c)  -- (d);
            \draw (c)  -- (e);
            \draw (d)  -- (e);
            \draw (e)  -- (f);
            \draw (0, -1) node[below]{0} -- (3.14, -1) node[below]{$\pi$}  -- (6.28, -1) node[below]{$2\pi$} -- (9.42, -1) node[below]{$3\pi$} -- (11,-1);
            \draw[style=help lines] (0,-1) -- (0,1);
            \draw[style=help lines] (3.14,-1) -- (3.14,1);
            \draw[style=help lines] (6.28,-1) -- (6.28,1);
            \draw[style=help lines] (9.42,-1) -- (9.42,1);
        \end{tikzpicture}
        \caption{} 
        \end{subfigure}
    \begin{subfigure}{0.4 \textwidth}
        \centering
        \begin{tikzpicture}[scale=0.5]
            \draw (3,0) node(a)[style=vertex]{};
            \draw (5.9,-0.5) node(b)[style=vertex]{};
            \draw (5.5,0) node(c)[style=vertex]{};
            \draw (6.1,0.5) node(d)[style=vertex]{};
            \draw (8.84,0) node(e)[style=vertex]{};
            \draw (0, 0) node(f)[style=vertex]{};
            \draw (2,0.5) node(g)[style=vertex]{};
            \draw (a)  -- (c);
            \draw (a)  -- (b);
            \draw (a)  -- (d);
            \draw (a)  -- (c);
            \draw (a)  -- (f);
            \draw (a)  -- (g);
            \draw (c)  -- (d);
            \draw (b)  -- (d);
            \draw (b)  -- (c);
            \draw (b)  -- (e);
            \draw (d)  -- (e);
            \draw[dashed] (f) -- (g);
            
            \draw (0, -1) node[below]{0} -- (3.14, -1) node[below]{$\pi$}  -- (6.28, -1) node[below]{$2\pi$} -- (9.42, -1) node[below]{$3\pi$};
            \draw[style=help lines] (0,-1) -- (0,1);
            \draw[style=help lines] (3.14,-1) -- (3.14,1);
            \draw[style=help lines] (6.28,-1) -- (6.28,1);
            \draw[style=help lines] (9.42,-1) -- (9.42,1);
        \end{tikzpicture}
        \caption{}
    \end{subfigure}
    \begin{subfigure}{0.4 \textwidth}
        \centering
        \begin{tikzpicture}[scale=0.5]
            \draw (4,0) node(a)[style=vertex]{};
            \draw (6.9,-0.5) node(b)[style=vertex]{};
            \draw (6.5,0) node(c)[style=vertex]{};
            \draw (7.1,0.5) node(d)[style=vertex]{};
            \draw (9.84,0) node(e)[style=vertex]{};
            \draw (0, 0) node(f)[style=vertex]{};
            \draw (3.5,0.5) node(g)[style=vertex]{};
            \draw (a)  -- (c);
            \draw (a)  -- (b);
            \draw (a)  -- (d);
            \draw (a)  -- (c);
            \draw[dashed] (a)  -- (f);
            \draw (a)  -- (g);
            \draw (c)  -- (d);
            \draw (b)  -- (d);
            \draw (b)  -- (c);
            \draw (b)  -- (e);
            \draw (d)  -- (e);
            
            \draw (0, -1) node[below]{0} -- (3.14, -1) node[below]{$\pi$}  -- (6.28, -1) node[below]{$2\pi$} -- (9.42, -1) node[below]{$3\pi$} -- (11,-1);
            \draw[style=help lines] (0,-1) -- (0,1);
            \draw[style=help lines] (3.14,-1) -- (3.14,1);
            \draw[style=help lines] (6.28,-1) -- (6.28,1);
            \draw[style=help lines] (9.42,-1) -- (9.42,1);
        \end{tikzpicture}
        \caption{}
    \end{subfigure}
    \begin{subfigure}{0.4 \textwidth}
        \centering
        \begin{tikzpicture}[scale=0.5]
            \draw (6.1,-0.5) node(a)[style=vertex]{};
            \draw (3,0.5) node(b)[style=vertex]{};
            \draw (5.9,0.5) node(c)[style=vertex]{};
            \draw (4,-0.5) node(d)[style=vertex]{};
            \draw (2,0) node(e)[style=vertex]{};
            \draw (9.2, 0) node(f)[style=vertex]{};
            \draw (0,0) node(g)[style=vertex]{};
            \draw (a)  -- (c);
            \draw (a)  -- (b);
            \draw (a)  -- (d);
            \draw (a)  -- (c);
            \draw (a)  -- (f);
            \draw (b)  -- (g);
            \draw (c)  -- (d);
            \draw (b)  -- (d);
            \draw (b)  -- (c);
            \draw (b)  -- (e);
            \draw (d)  -- (e);
            \draw[dashed] (g) -- (e);
            
            \draw (0, -1) node[below]{0} -- (3.14, -1) node[below]{$\pi$}  -- (6.28, -1) node[below]{$2\pi$} -- (9.42, -1) node[below]{$3\pi$};
            \draw[style=help lines] (0,-1) -- (0,1);
            \draw[style=help lines] (3.14,-1) -- (3.14,1);
            \draw[style=help lines] (6.28,-1) -- (6.28,1);
            \draw[style=help lines] (9.42,-1) -- (9.42,1);
        \end{tikzpicture}
        \caption{}
    \end{subfigure}
    \begin{subfigure}{0.4 \textwidth}
        \centering
        \begin{tikzpicture}[scale=0.5]
            \draw (7.1,-0.5) node(a)[style=vertex]{};
            \draw (4,0.5) node(b)[style=vertex]{};
            \draw (6.9,0.5) node(c)[style=vertex]{};
            \draw (4,-0.5) node(d)[style=vertex]{};
            \draw (3.2,0) node(e)[style=vertex]{};
            \draw (10.2, 0) node(f)[style=vertex]{};
            \draw (0,0) node(g)[style=vertex]{};
            \draw (a)  -- (c);
            \draw (a)  -- (b);
            \draw (a)  -- (d);
            \draw (a)  -- (c);
            \draw (a)  -- (f);
            \draw[dashed] (b)  -- (g);
            \draw (c)  -- (d);
            \draw (b)  -- (d);
            \draw (b)  -- (c);
            \draw (b)  -- (e);
            \draw (d)  -- (e);
            
            \draw (0, -1) node[below]{0} -- (3.14, -1) node[below]{$\pi$}  -- (6.28, -1) node[below]{$2\pi$} -- (9.42, -1) node[below]{$3\pi$}  -- (11,-1);
            \draw[style=help lines] (0,-1) -- (0,1);
            \draw[style=help lines] (3.14,-1) -- (3.14,1);
            \draw[style=help lines] (6.28,-1) -- (6.28,1);
            \draw[style=help lines] (9.42,-1) -- (9.42,1);
        \end{tikzpicture}
        \caption{}
    \end{subfigure}
    \caption{The left-hand side shows the constructions when removing an edge on the dotted line and the right-hand side shows the constructions when adding the edge indicated by the dotted line. }
    \label{fig:circulant_2}
\end{figure}
Again, since we may choose from infinitely many vertices at any stage, we can always choose one such that none of the edges/non-edges to the proceeding vertices have been altered.  
\subsection{The final graph}\label{sec:z3}
There is one final graph we need to handle, shown in Figure \ref{fig:final-graph}.
\begin{figure}[ht!]
    \centering
    \begin{tikzpicture}
        \draw (-1.5,0.5) node(0)[style=vertex]{};
        \draw (0.5, -0.5) node(1)[style=vertex]{};
        \draw (0,1) node(2)[style=vertex]{};
        \draw (-1.5,-0.5) node(3)[style=vertex]{};
        \draw (0.5, 0.5) node(4)[style=vertex]{};
        \draw (-0.5, -0.5) node(5)[style=vertex]{};
        \draw (-0.5, 0.5) node(6)[style=vertex]{};
        
        \draw (0) -- (5);
        \draw (0) -- (6);
        \draw (1) -- (4);
        \draw (1) -- (5);
        \draw (1) -- (6);
        \draw (2) -- (4);
        \draw (2) -- (6);
        \draw (3) -- (5);
        \draw (3) -- (6);
        \draw (4) -- (5);
        \draw (4) -- (6);
        \draw (5) -- (6);
        
    \end{tikzpicture}
    \caption{The final graph \texttt{F?q\textasciitilde{}w}.}
    \label{fig:final-graph}
\end{figure}
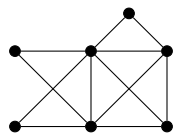

We first give an auxiliary construction and then blow this up to allow for locally finite perturbations. Let $G$ be the graph with vertex set $\mathbb{Z}^3$.  We join the vertex $(i,j, k)$ to  $(a,b,c)$ in $G$ if they agree in at least one coordinate, i.e. $i = a$, $j = b$ or $k = c$. There are two different kinds of edges $uv$ in $G$: those where $u$ and $v$ agree in exactly one coordinate and those where $u$ and $v$ agree in exactly two coordinates. There is one type of non-edge. 

The graph $G'$ that we are interested in is obtained by blowing up each vertex of $G$ into an infinite clique (replacing edges by complete bipartite graphs).
We can split the neighbours of any vertex of $G'$ into three cliques based on which coordinate they agree on in $G$, which shows that $G'$ does not contain a copy of $H$. 

Next, we show that making any locally finite perturbation creates a copy of $H$. We again first show that a single perturbation in $G$ results in a copy of $H$ in Figure \ref{fig:final-graph-constructions}.
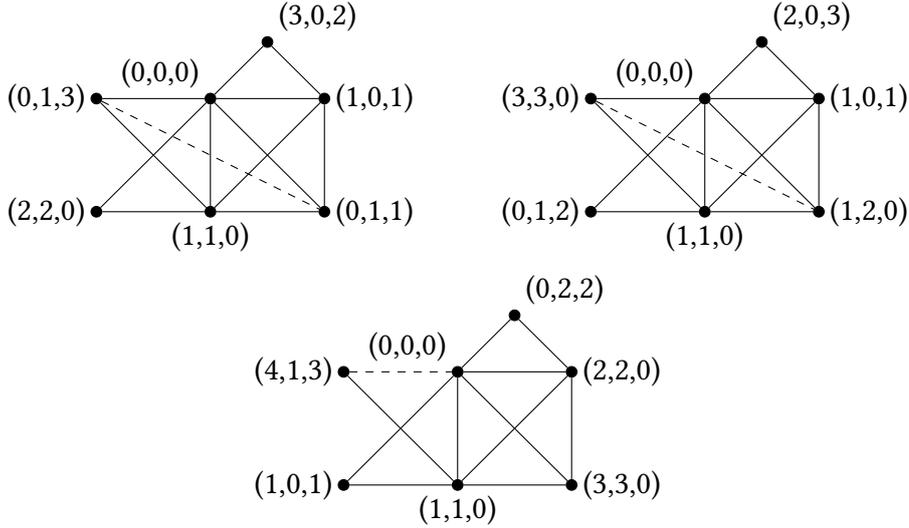
\begin{figure}[ht!]
    \centering
    \begin{subfigure}{0.4\textwidth}
    \begin{tikzpicture}[scale=1.5]
        \draw (-1.5,0.5) node(0)[style=vertex]{};
        \draw (-1.5,0.5) node[left]{(0,1,3)};
        \draw (0.5, -0.5) node(1)[style=vertex]{};
        \draw (0.5, -0.5) node[right]{(0,1,1)};
        \draw (0,1) node(2)[style=vertex]{};
        \draw (0,1) node[above right]{(3,0,2)};
        \draw (-1.5,-0.5) node(3)[style=vertex]{};
        \draw (-1.5,-0.5) node[left]{(2,2,0)};
        \draw (0.5, 0.5) node(4)[style=vertex]{};
        \draw (0.5, 0.5) node[right]{(1,0,1)};
        \draw (-0.5, -0.5) node(5)[style=vertex]{};
        \draw (-0.5, -0.5) node[below]{(1,1,0)};
        \draw (-0.5, 0.5) node(6)[style=vertex]{};
        \draw (-0.5, 0.5) node[above left]{(0,0,0)};
        
        \draw (0) -- (5);
        \draw (0) -- (6);
        \draw (1) -- (4);
        \draw (1) -- (5);
        \draw (1) -- (6);
        \draw (2) -- (4);
        \draw (2) -- (6);
        \draw (3) -- (5);
        \draw (3) -- (6);
        \draw (4) -- (5);
        \draw (4) -- (6);
        \draw (5) -- (6);
        \draw[dashed] (0) -- (1);
        
    \end{tikzpicture}
    \end{subfigure}
    \begin{subfigure}{0.4\textwidth}
    \begin{tikzpicture}[scale=1.5]
        \draw (-1.5,0.5) node(0)[style=vertex]{};
        \draw (-1.5,0.5) node[left]{(3,3,0)};
        \draw (0.5, -0.5) node(1)[style=vertex]{};
        \draw (0.5, -0.5) node[right]{(1,2,0)};
        \draw (0,1) node(2)[style=vertex]{};
        \draw (0,1) node[above right]{(2,0,3)};
        \draw (-1.5,-0.5) node(3)[style=vertex]{};
        \draw (-1.5,-0.5) node[left]{(0,1,2)};
        \draw (0.5, 0.5) node(4)[style=vertex]{};
        \draw (0.5, 0.5) node[right]{(1,0,1)};
        \draw (-0.5, -0.5) node(5)[style=vertex]{};
        \draw (-0.5, -0.5) node[below]{(1,1,0)};
        \draw (-0.5, 0.5) node(6)[style=vertex]{};
        \draw (-0.5, 0.5) node[above left]{(0,0,0)};
        
        \draw (0) -- (5);
        \draw (0) -- (6);
        \draw (1) -- (4);
        \draw (1) -- (5);
        \draw (1) -- (6);
        \draw (2) -- (4);
        \draw (2) -- (6);
        \draw (3) -- (5);
        \draw (3) -- (6);
        \draw (4) -- (5);
        \draw (4) -- (6);
        \draw (5) -- (6);
        \draw[dashed] (0) -- (1);
        
    \end{tikzpicture}
    \end{subfigure}
    \begin{subfigure}{0.4\textwidth}
    \begin{tikzpicture}[scale=1.5]
        \draw (-1.5,0.5) node(0)[style=vertex]{};
        \draw (-1.5,0.5) node[left]{(4,1,3)};
        \draw (0.5, -0.5) node(1)[style=vertex]{};
        \draw (0.5, -0.5) node[right]{(3,3,0)};
        \draw (0,1) node(2)[style=vertex]{};
        \draw (0,1) node[above right]{(0,2,2)};
        \draw (-1.5,-0.5) node(3)[style=vertex]{};
        \draw (-1.5,-0.5) node[left]{(1,0,1)};
        \draw (0.5, 0.5) node(4)[style=vertex]{};
        \draw (0.5, 0.5) node[right]{(2,2,0)};
        \draw (-0.5, -0.5) node(5)[style=vertex]{};
        \draw (-0.5, -0.5) node[below]{(1,1,0)};
        \draw (-0.5, 0.5) node(6)[style=vertex]{};
        \draw (-0.5, 0.5) node[above left]{(0,0,0)};
        
        \draw (0) -- (5);
        \draw[dashed] (0) -- (6);
        \draw (1) -- (4);
        \draw (1) -- (5);
        \draw (1) -- (6);
        \draw (2) -- (4);
        \draw (2) -- (6);
        \draw (3) -- (5);
        \draw (3) -- (6);
        \draw (4) -- (5);
        \draw (4) -- (6);
        \draw (5) -- (6);
        
    \end{tikzpicture}
    \end{subfigure}
    \caption{Examples of how to embed $H$ into a copy of $G$ when an edge of $G$ has been perturbed. The dotted lines represent the location of the edge which was removed (top) or added (bottom).}
    \label{fig:final-graph-constructions}
\end{figure}

Suppose now that $v_1'v_2'$ is perturbed in a locally finite perturbation of  $G'$ where $v_1'$ and $v_2'$ correspond to different vertices $v_1$ and $v_2$ in $G$. Then a copy of $H$ will still be created. The argument for this is similar to previous constructions: if $v_1,\dots,v_k$ are the vertices in $G$ on which a copy of $H$ is created after perturbing $v_1v_2$ in $G$, then we may iteratively choose $v_i'$ corresponding to $v_i$ such that the edges/non-edges from $v_i'$ to $v_1',\dots,v_{i-1}'$ have not been adjusted.

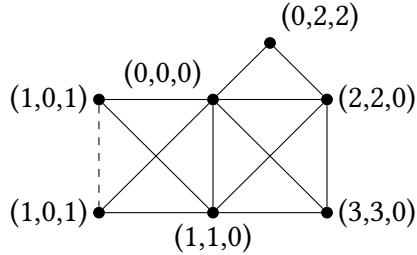
\begin{figure}[ht!]
\centering
    \begin{tikzpicture}[scale=1.5]
        \draw (-1.5,0.5) node(0)[style=vertex]{};
        \draw (-1.5,0.5) node[left]{(1,0,1)};
        \draw (0.5, -0.5) node(1)[style=vertex]{};
        \draw (0.5, -0.5) node[right]{(3,3,0)};
        \draw (0,1) node(2)[style=vertex]{};
        \draw (0,1) node[above right]{(0,2,2)};
        \draw (-1.5,-0.5) node(3)[style=vertex]{};
        \draw (-1.5,-0.5) node[left]{(1,0,1)};
        \draw (0.5, 0.5) node(4)[style=vertex]{};
        \draw (0.5, 0.5) node[right]{(2,2,0)};
        \draw (-0.5, -0.5) node(5)[style=vertex]{};
        \draw (-0.5, -0.5) node[below]{(1,1,0)};
        \draw (-0.5, 0.5) node(6)[style=vertex]{};
        \draw (-0.5, 0.5) node[above left]{(0,0,0)};
        
        \draw (0) -- (5);
        \draw (0) -- (6);
        \draw (1) -- (4);
        \draw (1) -- (5);
        \draw (1) -- (6);
        \draw (2) -- (4);
        \draw (2) -- (6);
        \draw (3) -- (5);
        \draw (3) -- (6);
        \draw (4) -- (5);
        \draw (4) -- (6);
        \draw (5) -- (6);
        \draw[dashed] (0) -- (3); 
        
    \end{tikzpicture}
    \caption{The construction shows how to obtain a copy of $H$ when an edge is removed between two vertices which come from the same vertex of $G$ (in this case, $(1,0,1)$).}
    \label{fig:final-graph-constructions2}
\end{figure}

We have also introduced a new type of edge in $G'$: an edge $u'u''$ where both $u'$ and $u''$ come from the infinite clique that corresponds to a single vertex $u$ of $G$. We indicate in Figure~\ref{fig:final-graph-constructions2} how to obtain a copy of $H$ in a locally finite perturbation that only affects edges of this type. 

\subsection{Proof of Theorem~\ref{thm:main}}\label{subsec:proof}
We now have all the components required to prove Theorem~\ref{thm:main}.

\begin{proof}[Proof of Theorem~\ref{thm:main}]
Suppose that $H$ is a finite graph on at least 12 vertices which is not a clique or independent set. 
By Theorem \ref{thm:cases}, either the graph $H$ or its complement $\overline{H}$ satisfies one of the following statements:
\begin{enumerate}
    \item $H$ is a forest with a unique vertex of maximum degree,
    \item the $2$-core of $H$ is $K_{2,p}$ for $p\geq 3$,
    \item the $2$-core of $H$ is $K_{1,1,p}$ for $p\geq 3$, or 
    \item the $3^*$-core of $H$ is 3-connected and not a clique.
\end{enumerate}
In the first and fourth cases, there exists a strongly $H$-induced-saturated graph by Lemma~\ref{lem:forest} and Corollary~\ref{cor:3starcore}, respectively. By Lemma~\ref{lem:K2p}, any graph in the second case admits a fixing operation for the class of $K_{2,p}$-free graphs. By Lemma~\ref{lem:K11p}, any graph in the third case admits a fixing operation for a particular class of $K_{1,1,p}$-free graphs. Applying Lemma~\ref{lem:fix_to_strong} in each of these cases, and noting that the complement of any strongly $H$-induced-saturated graph is a strongly $\overline{H}$-induced-saturated graph, completes the proof for any $H$ on at least 12 vertices.

For graphs on at most 11 vertices, our computer search identifies whether each graph $H$ or its complement satisfies any of the following conditions. 

\begin{enumerate}
    \item The graph $H$ is a non-empty forest with a unique vertex of maximum degree.
    \item The 2-core of $H$ is a copy $K_{2,p}$ with $p \geq 3$.
    \item The 2-core of $H$ is a copy of $K_{1,1,p}$ where $p \geq 2$.
    \item The $3$-core of $H$ is 3-connected and not a clique.
    \item The $(1,1)$-core of $H$ is a copy of $P_4$ or the bull graph.
    \item The graph $H$ is close to a permutation graph.
    \item The $2$-core, $3$-core, $2$-edge-core or $2$-non-edge-core have fixing operations.
\end{enumerate}
The first four cases follow for the same reasons they did for graphs on at least 12 vertices. The fifth and sixth cases are resolved by Theorems~\ref{thm:bull} and \ref{thm:upandright}, respectively. The last case is considered in \Cref{sec:computermagic}.
    There are 8 graphs (\texttt{E?qw}, \texttt{F?S|w}, \texttt{F?q|w} and \texttt{F?q~w}, and their complements) which do not fall into one of the above cases. The first six of these are resolved in Section~\ref{sec:circ}, and the final pair is resolved in Section~\ref{sec:z3}.
\end{proof}

\section{Open problems}\label{sec:ccl}
We have shown that a finite graph $H$ admits a countable $H$-induced-saturated graph if and only if $H$ is not a clique or independent set, and in fact the characterisation remains the same when asking for a much stronger notion of saturation. 

In the finite case, the picture is less clear. Indeed, the classification of the existence of finite $H$-induced-saturated graphs is still widely open. We do not know of any other example of a finite graph $H$, besides $P_4$, independent sets and cliques, for which no finite $H$-induced-saturated graph exists. Since cliques and independent sets are in some sense ``trivial'' examples for which no $H$-induced-saturated graph can exist, it would be very interesting to know if there are infinitely many non-trivial examples.
\begin{problem}
    Is there an infinite family $\mathcal{G}$ of finite graphs, not containing cliques or independent sets, such that for all $H\in \mathcal{G}$, there is no finite $H$-induced-saturated graph?
\end{problem}
An $H$-induced-saturated graph on $n$ vertices can be seen as an isolated vertex in the class of $H$-free graphs (where edges are placed between graphs at edit-distance 1). Instead of finding an isolated vertex, it may be easier to show that this graph is disconnected.
\begin{problem}
\label{prob:disconnected}
    For a finite graph $H$ and integer $n$, let $G_{H,n}$ denote the graph with the $H$-free graphs on vertex set $[n]$ as vertices, and with an edge between $G,G'\in V(G_{H,n})$ if and only if $|E(G)\triangle E(G')|=1$. For which $H$ is $G_{H,n}$ disconnected for all sufficiently large $n$?
\end{problem}
We remark that, from the inductive cograph definition, it is not too difficult to show that $G_{P_4,n}$ is connected (for all choices of $n\geq 1$). 

Many interesting directions in the infinite world also remain open. 
What about structures other than graphs? For example, with an appropriate notion of ``perturbation'', can a similar (or partial) characterisation be obtained for $k$-uniform hypergraphs, coloured graphs, matroids, posets or directed graphs?  For tournaments the natural operation is to reverse the direction of edges, which leads to the following problem.
\begin{problem}
    For which finite tournaments $H$, does there exist a countable tournament $G$ that does not contain $H$ as subtournament, yet any tournament obtained from $G$ by changing the direction of a single arc does contain $H$? 
\end{problem}
For example, when $H$ is a directed triangle then there is no finite example for $G$, but an infinite example can be given by taking the graph on vertex set $\mathbb{Q}$ where $q\to q'$ if and only if $q<q'$. For the transitive tournament on $3$ elements, the directed triangle is a finite example.

Perhaps the following analogue of Theorem \ref{thm:main} is true for the ``locally finite perturbation'' variant of this question.
\begin{conjecture}
Let $T$ be a finite tournament. If $T$ is not transitive, then there exists a countably infinite $T$-free tournament $S_T$ such that every locally finite perturbation of $S_T$ has an induced copy of $T$.
\end{conjecture}
Note that the locally finite perturbations enforce that $S_T$ is infinite. As any tournament on $2^n$ elements has a transitive subtournament on $n$ elements\footnote{This can be seen by greedily constructing such a transitive subtournament using the fact that each vertex has either at least half of the (remaining) vertices in its in-neighbourhood, or in its out-neighbourhood.}, the requirement that $T$ is not transitive is necessary. 

Similar questions arise for $k$-uniform hypergraphs, where a single perturbation is the addition or removal of an edge. Here, the strongest conjecture is the following.
\begin{conjecture}
Let $k\ge3$ be an integer, and let $H$ be a finite $k$-uniform hypergraph. If $H$ has both an edge and a non-edge, then there exists a countably infinite $k$-uniform hypergraph $G_H$ such that every locally finite perturbation of $G_H$ has an induced copy of $H$.
\end{conjecture}

It would also be interesting to prove analogous results for coloured structures.  Let $k\ge2$ be an integer, and consider a $k$-colouring $c:E(K)\to[k]$ of the edges of a (finite or infinite) complete graph $K$.  A {\em locally finite perturbation} of $c$ is a colouring with colours on some (nonempty) locally finite subgraph of $G$ changed.
\begin{conjecture}
Let $k\ge2$ be an integer, and let $c$ be a $k$-colouring
of the edges of a finite clique $K$ such that there is an edge of every colour.  Then there is a colouring $\chi$ of the edges of the complete countably infinite graph such that every locally finite perturbation contains a copy of $K$ with colouring $c$.
\end{conjecture}
Another direction is to show the existence (or non-existence) of graphs that are simultaneously (strongly) $H$-induced-saturated for all graphs $H$ in a specific class of graphs. We gave such a result for the class of forests which have a unique vertex of maximum degree, and also characterised in Theorem \ref{thm:upandright} for which graphs $H$ our ``up and right graph'' is strongly $H$-induced-saturated. What about the class of finite cographs?  Or finite graphs of bounded twinwidth? 

Finally, there could be analogues of our result for infinite $H$, for example, if $H$ is a graph of cardinality $\kappa$, when are there $H$-free graphs of cardinality $\lambda$ that satisfy our theorem with locally $\kappa$ perturbations? Could this be true whenever $H$ is neither complete nor empty, and $\lambda$ is larger than $\kappa$?

\paragraph{Acknowledgements.} We would like to thank Mathieu Rundstrom for helpful comments and questions that fixed some issues in an earlier version of this paper, Reinhard Diestel for helpful comments and Torsten Ueckerdt for suggesting Problem~\ref{prob:disconnected} and allowing us to include this here.

\bibliographystyle{style}
\bibliography{refs}

\end{document}